\documentclass[11pt,a4paper, final, twoside]{article}
\usepackage{amsmath}
\usepackage{fancyhdr}
\usepackage{amsthm}
\usepackage{amsfonts}
\usepackage{amssymb}
\usepackage{amscd}
\usepackage{amsthm}
\usepackage{graphicx}
\usepackage{afterpage}
\usepackage[colorlinks=true, urlcolor=blue,  linkcolor=blue, citecolor=blue]{hyperref}

\newtheorem{theorem}{Theorem}

\newtheorem{corollary}[theorem]{Corollary}

\newtheorem{definition}[theorem]{Definition}

\newtheorem{lemma}[theorem]{Lemma}

\newtheorem{proposition}[theorem]{Proposition}
\newtheorem{remark}[theorem]{Remark}

\numberwithin{equation}{section}

\begin{document}

\hyphenpenalty=10000

\begin{center}
{\Large \textbf{SPECTRUM OF A FAMILY OF OPERATORS }}\\[5mm]
{\large {Simona Macovei }\\[10mm]
}
\end{center}

{\footnotesize \textbf{Abstract}. Having as start point the classic
definitions of resolvent set and spectrum of a linear bounded operator on a
Banach space, we introduce the resolvent set and spectrum of a family of
linear bounded operators on a Banach space. In addition, we present some
results which adapt to asymptotic case the classic results.} \footnote{%
} \footnote{%
}

\fancyhead{} \fancyfoot{} 
\fancyhead[LE, RO]{\bf\thepage}
\fancyhead[LO]{\small Spectrum of a Family of Operators}
\fancyhead[RE]{\small Simona Macovei  }

\section{Introduction}

\noindent

\noindent Let \textit{X} be a complex Banach space and $L(X)$ the Banach
algebra of linear bounded operators on \textit{X}. Let \textit{T} be a
linear bounded operator on \textit{X}. The\textit{\ norm} of \textit{T} is

\noindent 
\begin{equation*}
\left\|T\right\|=sup\left\{\left\|Tx\right\||\ x\in X,\ \left\|x\right\|\le
1\right\}.
\end{equation*}

\noindent The\textit{\ spectrum} of an operator $T\in L(X)$ is defined as
the set

\noindent 
\begin{equation*}
Sp\left(T\right)\mathrm{=}{\mathbb{C}}\backslash r(T),
\end{equation*}

\noindent where $r(T)$ is the \textit{resolvent set} of \textit{T} and
consists in all complex numbers $\lambda \in {\mathbb{C}}$ for which the
operator $\lambda I-T$ is bijective on \textit{X}.

\noindent It is an important fact that the \textit{resolvent} \textit{%
function }$\lambda \mapsto {(\lambda I-T)}^{-1}$ is an analytic function
from $r(T)$ to $L(X)$ and for $\lambda \in r(T)$ we have

\noindent 
\begin{equation*}
d\left(\lambda ,r\left(T\right)\right)\ge \frac{1}{\left\|{(\lambda I-T)}%
^{-1}\right\|}.
\end{equation*}

\noindent Moreover, for $\lambda \in r(T)$, the\textit{\ resolvent operator }%
$R\left(\lambda ,T\right)\in L(X)$ is defined by the relation $%
R\left(\lambda ,T\right)={(\lambda I-T)}^{-1}$ and satisfied the\textit{\
resolvent equation }

\noindent 
\begin{equation*}
R\left(\lambda ,T\right)-R\left(\mu ,T\right)=(\mu -\lambda )R\left(\lambda
,T\right)R\left(\mu ,T\right),
\end{equation*}

\noindent for all $\lambda ,\mu \in r(T)$. Therefore, in particular, $%
R\left(\lambda ,T\right)$ and $R\left(\mu ,T\right)$ commute.

\noindent We say that an infinite series of operators $\sum{T_n}$ is
absolutely convergent if the series $\sum{\left\|T_n\right\|}$ is convergent
in $L(X)$ and $\left\|\sum{T_n}\right\|\le \sum{\left\|T_n\right\|}$.

\noindent If $\left\|T\right\|<1$, then 
\begin{equation*}
{(\lambda I-T)}^{-1}=\sum{T^n}
\end{equation*}

\noindent and it is absolutely convergent. A consequence of this is the fact
that $r(T)$ is an open set of ${\mathbb{C}}$.

\noindent

\begin{theorem}
\noindent \label{t1.1}Theorem 1.1. Let $T\in L(X)$ be a linear bounded
operator on X. Then $Sp\left( T\right) $ is a non-empty compact subset of $C$%
.

The spectral radius of an operator $T\in L(X)$ is the positive number equal
with $\mathop{sup}_{\lambda \in Sp(T)}\left\vert \lambda \right\vert \ $.
\end{theorem}

\noindent 

\begin{theorem}
\noindent \label{t1.2}Let $T\in L(X)$. Then 
\begin{equation*}
{\mathop{sup}_{\lambda \in Sp(T)}\left\vert \lambda \right\vert \ }={{%
\mathop{\lim }_{n\rightarrow \infty }}{\left\Vert T^{n}\right\Vert }^{\frac{1%
}{n}}\ }.
\end{equation*}

\noindent Let $\Omega $ be an open neighborhood of $Sp\left( T\right) $ and
let $H(\Omega )$ denote the space of all complex valued analytic functions
defined on $\Omega $. The application $f\mapsto f(T):H(\Omega )\rightarrow
L(X)$ defined by the relation 
\begin{equation*}
f\left( T\right) =\frac{1}{2\pi i}\int_{\gamma }{f\left( \lambda \right)
R\left( \lambda ,T\right) d\lambda },
\end{equation*}

\noindent where $\gamma $ is a contour which envelopes $Sp(T)$ in $\Omega $,
is called the holomorphic functional calculi of $T$.
\end{theorem}

\noindent 

\begin{theorem}
\label{t1.3}\noindent Let $T\in L(X)$ and suppose that $\Omega $ is an open
neighborhood of $Sp\left( T\right) $. Then, for all $f\in H(\Omega )$, we
have 
\begin{equation*}
f\left( Sp(T)\right) =Sp(f\left( T\right) ).
\end{equation*}
\end{theorem}

\noindent We also remember that two operators $T,S\in L(X)$ are \textit{%
quasinilpotent equivalent} if

\begin{equation*}
{\mathop{\lim }_{n\rightarrow \infty }{\left\Vert {(T-S)}^{\left[ n\right]
}\right\Vert }^{\frac{1}{n}}\ }={\mathop{\lim }_{n\rightarrow \infty }{%
\left\Vert {(S-T)}^{\left[ n\right] }\right\Vert }^{\frac{1}{n}}\ }=0,
\end{equation*}

\noindent where ${\left(T-S\right)}^{\left[n\right]}=\sum^n_{k=0}{{%
\left(-1\right)}^{n-k}C^n_kT^kS^{n-k}}$, for any $n\in {\mathbb{N}}$.

\noindent The quasinilpotent equivalence relation is reflexive and
symmetric. It is also transitive on $L(X)$.

\noindent

\begin{theorem}
\noindent \label{t1.4}Theorem 1.4. Let $T,S\in L(X)$ be two quasinilpotent
equivalent operators. Then%
\begin{equation*}
Sp\left( T\right) =Sp\left( S\right) \text{.}
\end{equation*}
\end{theorem}

\section{Asymptotic equivalence and asymptotic quasinilpotent equivalence}

\noindent

\begin{definition}
\label{d2.1}We say that two families of operators $\left\{ S_{h}\right\} ,\
\left\{ T_{h}\right\} \ \subset L(X)$, with $h\in \left. (0,1\right] ,$ are
asymptotically equivalent if%
\begin{equation*}
{\mathop{lim}_{h\rightarrow 0}\left\Vert S_{h}-T_{h}\right\Vert =0\ }.
\end{equation*}

\noindent Two families of operators $\left\{ S_{h}\right\} ,\ \left\{
T_{h}\right\} \ \subset L(X)$, with $h\in \left. (0,1\right] $, are
asymptotically quasinilpotent equivalent if\noindent 
\begin{equation*}
{\mathop{\lim}_{n\rightarrow \infty }{{\mathop{\lim \sup}_{h\rightarrow
0}\left\Vert {\left( S_{h}-T_{h}\right) }^{\left[ n\right] }\right\Vert \ }}%
^{\frac{1}{n}}\ }={\mathop{\lim}_{n\rightarrow \infty }{{\mathop{\lim \sup}%
_{h\rightarrow 0}\left\Vert {\left( T_{h}-S_{h}\right) }^{\left[ n\right]
}\right\Vert \ }}^{\frac{1}{n}}\ }=0.
\end{equation*}
\end{definition}

\begin{proposition}
\noindent \label{p2.2}The asymptotic (quasinilpotent) equivalence between
two families of operators $\left\{ S_{h}\right\} ,\ \left\{ T_{h}\right\}
\subset L(X)$ is an equivalence relation (i.e. reflexive, symmetric and
transitive) on $L\left( X\right) $.
\end{proposition}

\noindent

\begin{proof}
It is evidently that the asymptotic equivalence is reflexive and symmetric.

\noindent Let $\left\{ S_{h}\right\} ,\ \left\{ T_{h}\right\} ,\ \left\{
U_{h}\right\} \subset L(X)$ be families of linear bounded operators such
that $\left\{ S_{h}\right\} ,\ \left\{ T_{h}\right\} $ and $\left\{
U_{h}\right\} ,\ \left\{ T_{h}\right\} $ are respectively asymptotically
equivalent. Then\noindent 
\begin{align*}
& {\mathop{\lim \sup}_{h\rightarrow 0}\left\Vert S_{h}-U_{h}\right\Vert } \\
& ={{\mathop{\lim \sup}_{h\rightarrow 0}\left\Vert
S_{h}-T_{h}+T_{h}-U_{h}\right\Vert \ }\leq {\mathop{\lim }_{h\rightarrow
0}\left\Vert S_{h}-T_{h}\right\Vert +{\mathop{\lim }_{h\rightarrow
0}\left\Vert T_{h}-U_{h}\right\Vert \ }}} \\
{}& {{=\ }0}\text{{\ }.}
\end{align*}

\noindent The asymptotic quasinilpotent equivalence is also reflexive and
symmetric.

\noindent In order to prove that it is transitive, let $\left\{
S_{h}\right\} ,\ \left\{ T_{h}\right\} ,\ \left\{ P_{h}\right\} \subset L(X)$
such that $\left\{ T_{h}\right\} ,\ \left\{ P_{h}\right\} $ and $\left\{
S_{h}\right\} ,\ \left\{ P_{h}\right\} $ be respectively asymptotically
quasinilpotent equivalent. Then for any $\varepsilon $ $>$ 0 there exists a $%
n_{\varepsilon }\in {\mathbb{N}}$ such that%
\begin{equation*}
\left( T_{h}-P_{h}\right) ^{\left[ j\right] }<{\varepsilon }^{j}
\end{equation*}%
and 
\begin{equation*}
{\left( P_{h}-S_{h}\right) }^{\left[ n-j\right] }<{\varepsilon }^{n-j},
\end{equation*}

\noindent \noindent for every $j,n-j>n_{\varepsilon }$ and $h\in \left. (0,1%
\right] $.

\noindent Taking 
\begin{equation*}
M_{\varepsilon }={\mathop{\max }_{1\leq j\leq n_{\varepsilon }}\left\{ \frac{%
\left\Vert {\left( T_{h}-P_{h}\right) }^{\left[ j\right] }\right\Vert }{{%
\varepsilon }^{j}},\frac{\left\Vert {\left( P_{h}-S_{h}\right) }^{\left[ j%
\right] }\right\Vert }{{\varepsilon }^{j}},1\right\} \ }
\end{equation*}%
we obtain%
\begin{equation*}
\left\Vert {\left( T_{h}-P_{h}\right) }^{\left[ j\right] }\right\Vert <{%
\varepsilon }^{j}M_{\varepsilon }
\end{equation*}%
and 
\begin{equation*}
\left\Vert {\left( P_{h}-S_{h}\right) }^{\left[ j\right] }\right\Vert <{%
\varepsilon }^{j}M_{\varepsilon },
\end{equation*}

for every $j\in {\mathbb{N}}$ and $h\in \left. (0,1\right] .$

\noindent In view of above inequality and the following equality\noindent 
\begin{equation*}
{(T-S)}^{\left[ n\right] }=\sum_{k=0}^{n}{{\left( -1\right) }^{n-k}C_{n}^{k}{%
\left( T-P\right) }^{\left[ k\right] }{\left( P-S\right) }^{\left[ n-k\right]
}},
\end{equation*}

for every\textit{\ }$n\in {\mathbb{N}}$\textit{\ }and\textit{\ }$P\in L(X)$,
it results that\noindent 
\begin{eqnarray*}
\left\Vert {\left( T_{h}-S_{h}\right) }^{\left[ n\right] }\right\Vert  &\leq
&\sum_{k=0}^{n}{C_{n}^{k}\left\Vert {\left( T_{h}-P_{h}\right) }^{\left[ k%
\right] }\right\Vert \left\Vert {\left( P_{h}-S_{h}\right) }^{\left[ n-k%
\right] }\right\Vert } \\
&\leq &\sum_{k=0}^{n}{C_{n}^{k}{\varepsilon }^{k}{\varepsilon }^{n-k}{%
M_{\varepsilon }}^{2}} \\
&=&{(2\varepsilon )}^{n}{M_{\varepsilon }}^{2}\text{,}
\end{eqnarray*}

{for every $n\in {\mathbb{N}}$ and $h\in \left. (0,1\right] $. }

\noindent Therefore 
\begin{equation*}
{\mathop{\lim \sup}_{h\to 0} \left\|{\left(T_h-S_h\right)}^{\left[n\right]%
}\right\|\ }\le {(2\varepsilon )}^n{M_{\varepsilon }}^2
\end{equation*}
and thus 
\begin{equation*}
{{\mathop{\lim \sup}_{h\to 0} \left\|{\left(T_h-S_h\right)}^{\left[n\right]%
}\right\|\ }}^{\frac{1}{n}}\le 2\varepsilon {M_{\varepsilon }}^{2/n}.
\end{equation*}
Consequently 
\begin{equation*}
{\mathop{\lim }_{n\to \infty } {{\mathop{\lim \sup}_{h\to 0} \left\|{%
\left(T_h-S_h\right)}^{\left[n\right]}\right\|\ }}^{\frac{1}{n}}\ }\le
2\varepsilon ,
\end{equation*}
for any $\varepsilon $ $>$ 0.

\noindent Analogously we prove that ${\mathop{\lim }_{n\to \infty } {{%
\mathop{\lim \sup}_{h\to 0} \left\|{\left(S_h-T_h\right)}^{\left[n\right]%
}\right\|\ }}^{\frac{1}{n}}\ }=0$.
\end{proof}

\begin{proposition}
\label{p2.3}Let $\left\{ S_{h}\right\} ,\ \left\{ T_{h}\right\} \subset L(X)$
be asymptotically equivalent.

\begin{description}
\item[i)] If $\left\{ S_{h}\right\} $ is a bounded family of operators, then 
$\left\{ T_{h}\right\} $ is also bounded and conversely;

\item[ii)] $\left\{ S_{h}\right\} ,\ \left\{ T_{h}\right\} $ are asymptotically
commuting (i.e. $\mathop{lim}_{h\rightarrow 0}\left\Vert S_{h}T_{h}-{T_{h}S}%
_{h}\right\Vert =0\ $);

\item[iii)] Let $\left\{ U_{h}\right\} \subset L(X)$ be a bounded family of
operators such that $\mathop{lim}_{h\rightarrow 0}\left\Vert S_{h}U_{h}-{%
U_{h}S}_{h}\right\Vert =0\ $. Then $\mathop{lim}_{h\rightarrow 0}\left\Vert
U_{h}T_{h}-{T_{h}U}_{h}\right\Vert =0\ $.
\end{description}
\end{proposition}

\begin{proof}
i) If $\left\{ S_{h}\right\} $ is a bounded family of operators, then there
is ${\mathop{\lim \sup }_{h\rightarrow 0}\left\Vert S_{h}\right\Vert \ }%
<\infty $. Since 
\begin{equation*}
{\mathop{\lim }_{h\rightarrow 0}\left\Vert S_{h}-T_{h}\right\Vert =0\ },
\end{equation*}%
it follows that\noindent 
\begin{eqnarray*}
{\mathop{\lim \sup }_{h\rightarrow 0}\left\Vert T_{h}\right\Vert \ } &=&{%
\mathop{\lim \sup }_{h\rightarrow 0}\left\Vert T_{h}-S_{h}+S_{h}\right\Vert
\ } \\
&\leq &{\mathop{\lim }_{h\rightarrow 0}\left\Vert S_{h}-T_{h}\right\Vert \ }+%
{\mathop{\lim \sup }_{h\rightarrow 0}\left\Vert S_{h}\right\Vert \ }<\infty .
\end{eqnarray*}

Therefore $\left\{ T_{h}\right\} $ is a bounded family of operators.

\noindent Analogously we can prove that if $\left\{ T_{h}\right\} $ is a
bounded family of operators, than $\left\{ S_{h}\right\} $ is a bounded
family of operators.

ii) 
\begin{equation*}
{\mathop{\lim \sup}_{h\rightarrow 0}\left\Vert S_{h}T_{h}-{T_{h}S}%
_{h}\right\Vert \ }={\mathop{\lim \sup}_{h\rightarrow 0}\left\Vert S_{h}T_{h}-{{{%
S_{h}}^{2}+{S_{h}}^{2}-T}_{h}S}_{h}\right\Vert \ }\leq 
\end{equation*}%
\begin{eqnarray*}
&\leq &{\mathop{\lim \sup}_{h\rightarrow 0}\left\Vert S_{h}\left(
S_{h}-T_{h}\right) \right\Vert \ }+{\mathop{\lim \sup}_{h\rightarrow
0}\left\Vert \left( S_{h}-T_{h}\right) S_{h}\right\Vert \ } \\
&\leq &{\mathop{{\rm 2\lim\sup}}_{h\rightarrow 0}\left\Vert S_{h}\right\Vert
\left\Vert S_{h}-T_{h}\right\Vert \ }\leq 0.
\end{eqnarray*}%
iii) 
\begin{equation*}
{\mathop{\lim \sup}_{h\rightarrow 0}\left\Vert T_{h}U_{h}-{U_{h}T}%
_{h}\right\Vert \ }={\mathop{\lim \sup}_{h\rightarrow 0}\left\Vert
T_{h}U_{h}-S_{h}U_{h}+S_{h}U_{h}-{U_{h}S}_{h}+{U_{h}S}_{h}-{U_{h}T}%
_{h}\right\Vert \ }\leq 
\end{equation*}%
\begin{equation*}
\leq {\mathop{\lim \sup}_{h\rightarrow 0}\left\Vert
T_{h}U_{h}-S_{h}U_{h}\right\Vert \ }+{\mathop{\lim \sup}_{h\rightarrow
0}\left\Vert S_{h}U_{h}-{U_{h}S}_{h}\right\Vert \ }+{\mathop{\lim \sup}%
_{h\rightarrow 0}\left\Vert {U_{h}S}_{h}-{U_{h}T}_{h}\right\Vert \ }\leq 
\end{equation*}%
\begin{equation*}
\leq \mathrm{2}{\mathop{\lim \sup}_{h\rightarrow 0}\left\Vert U_{h}\right\Vert
\left\Vert T_{h}-S_{h}\right\Vert }
\end{equation*}

Since $\left\{ U_{h}\right\} $ is a bounded family of operators, then there
is ${\mathop{\lim \sup }_{h\rightarrow 0}\left\Vert U_{h}\right\Vert \ }%
<\infty $. So%
\begin{equation*}
{\mathop{\lim }_{h\rightarrow 0}\left\Vert U_{h}T_{h}-{T_{h}U}%
_{h}\right\Vert =0}\text{.}
\end{equation*}
\end{proof}

\begin{proposition}
\label{p2.4}Let $\left\{ S_{h}\right\} ,\ \left\{ T_{h}\right\} \subset L(X)$
be two bounded families of operators such that $\mathop{lim}_{h\rightarrow
0}\left\Vert S_{h}T_{h}-T_{h}S_{h}\right\Vert \ =0$. Then

\begin{description}
\item[i)] ${\mathop{lim}_{h\rightarrow 0}\left\Vert
S_{h}^{n}T_{h}^{m}-T_{h}^{m}S_{h}^{n}\right\Vert \ }=0$, for any $n,m\in {%
\mathbb{N}}$;

\item[ii)] ${\mathop{lim}_{h\rightarrow 0}\left\Vert {\left(
S_{h}-T_{h}\right) }^{\left[ n\right] }\right\Vert \ }={\mathop{lim}%
_{h\rightarrow 0}\left\Vert {\left( S_{h}-T_{h}\right) }^{n}\right\Vert \ }$%
, for any $n\in {\mathbb{N}}$;

\item[iii)] ${\mathop{lim}_{h\rightarrow 0}\left\Vert {(S_{h}T_{h})}%
^{n}-S_{h}^{n}T_{h}^{n}\right\Vert \ }=0$, for any $n\in {\mathbb{N}}$.
\end{description}
\end{proposition}

\begin{proof}
i) We prove that ${\mathop{\lim }_{h\rightarrow 0}\left\Vert
S_{h}^{n}T_{h}-T_{h}S_{h}^{n}\right\Vert \ }=0$, for any $n\in {\mathbb{N}}$%
. For $n=2$ we have\noindent 
\begin{equation*}
{\mathop{\lim \sup}_{h\rightarrow 0}\left\Vert
S_{h}^{2}T_{h}-T_{h}S_{h}^{2}\right\Vert \ }={\mathop{\lim \sup}_{h\rightarrow
0}\left\Vert S_{h}\left( S_{h}T_{h}\right) -S_{h}\left( T_{h}S_{h}\right)
+\left( S_{h}T_{h}\right) S_{h}-(T_{h}S_{h})S_{h}\right\Vert \ }\leq 
\end{equation*}%
\begin{equation*}
\leq 2{\mathop{\lim \sup}_{h\rightarrow 0}\left\Vert \left( S_{h}T_{h}\right)
-\left( T_{h}S_{h}\right) \right\Vert \left\Vert S_{h}\right\Vert \ }=0.
\end{equation*}%
For $n=3$ we have%
\begin{equation*}
{\mathop{\lim \sup}_{h\rightarrow 0}\left\Vert
S_{h}^{3}T_{h}-T_{h}S_{h}^{3}\right\Vert }={\mathop{\lim \sup}_{h\rightarrow
0}\left\Vert S_{h}\left( S_{h}^{2}T_{h}\right) -S_{h}\left(
T_{h}S_{h}^{2}\right) +\left( S_{h}T_{h}\right)
S_{h}^{2}-(T_{h}S_{h})S_{h}^{2}\right\Vert \ }\leq 
\end{equation*}%
\begin{equation*}
\leq {\mathop{\lim \sup}_{h\rightarrow 0}\left\Vert
S_{h}^{2}T_{h}-T_{h}S_{h}^{2}\right\Vert \left\Vert S_{h}\right\Vert \ }+{%
\mathop{\lim }_{h\rightarrow 0}\left\Vert S_{h}T_{h}-T_{h}S_{h}\right\Vert
\left\Vert S_{h}^{2}\right\Vert \ }=0.
\end{equation*}

Considering relation ${\mathop{\lim }_{h\rightarrow 0}\left\Vert
S_{h}^{n}T_{h}-T_{h}S_{h}^{n}\right\Vert \ }=0$ true we prove that 
\begin{equation*}
{\mathop{\lim }_{h\rightarrow 0}\left\Vert
S_{h}^{n+1}T_{h}-T_{h}S_{h}^{n+1}\right\Vert \ }=0.
\end{equation*}%
\begin{equation*}
{\mathop{\lim \sup}_{h\rightarrow 0}\left\Vert
S_{h}^{n+1}T_{h}-T_{h}S_{h}^{n+1}\right\Vert \ }={\mathop{\lim \sup}%
_{h\rightarrow 0}\left\Vert S_{h}\left( S_{h}^{n}T_{h}\right) -S_{h}\left(
T_{h}S_{h}^{n}\right) +\left( S_{h}T_{h}\right)
S_{h}^{n}-(T_{h}S_{h})S_{h}^{n}\right\Vert \ }\leq 
\end{equation*}%
\begin{equation*}
\leq {\mathop{\lim \sup}_{h\rightarrow 0}\left\Vert
S_{h}^{n}T_{h}-T_{h}S_{h}^{n}\right\Vert \left\Vert S_{h}\right\Vert \ }+{%
\mathop{\lim \sup}_{h\rightarrow 0}\left\Vert S_{h}T_{h}-T_{h}S_{h}\right\Vert
\left\Vert S_{h}^{n}\right\Vert \ }=0.
\end{equation*}

Applying above relation to $S_{h}^{n}$ and $T_{h}$, it follows that\noindent 
\begin{equation*}
{\mathop{\lim }_{h\rightarrow 0}\left\Vert
S_{h}^{n}T_{h}^{m}-T_{h}^{m}S_{h}^{n}\right\Vert \ }=0,
\end{equation*}%
for every $n,m\in {\mathbb{N}}$.

ii) and iii) can be proved analogously i).
\end{proof}

\begin{proposition}
\label{p2.5} Let $\left\{ S_{h}\right\} ,\ \left\{ T_{h}\right\} \ \subset
L\left( X\right) $ be two bounded families of operators.

\begin{description}
\item[i)] If $\left\{ S_{h}\right\} ,\ \left\{ T_{h}\right\} \ $ are
asymptotically equivalent, then are asymptotically quasinilpotent equivalent.

\item[ii)] If $\mathop{\lim }_{h\rightarrow 0}\left\Vert
S_{h}T_{h}-T_{h}S_{h}\right\Vert \ =0$ and $\mathop{\lim }_{n\rightarrow
\infty }\mathop{\lim \sup}_{h\rightarrow 0}\left\Vert {\left( S_{h}-T_{h}\right) 
}^{\left[ n\right] }\right\Vert \ ^{\frac{1}{n}}\ =0$, then $\left\{
S_{h}\right\} ,\ \left\{ T_{h}\right\} $ are asymptotically quasinilpotent
equivalent, i.e.%
\begin{equation*}
{\mathop{\lim }_{n\rightarrow \infty }{{\mathop{\lim \sup}_{h\rightarrow
0}\left\Vert {\left( T_{h}-S_{h}\right) }^{\left[ n\right] }\right\Vert \ }}%
^{\frac{1}{n}}\ }=0.
\end{equation*}

\item[iii)] Let $\left\{ A_{h}\right\} \subset L(X)$ be a bounded families
of operators. If $\left\{ S_{h}\right\} ,\ \left\{ T_{h}\right\} $ are
asymptotically quasinilpotent equivalent and $\mathop{\lim }_{h\rightarrow
0}\left\Vert S_{h}A_{h}-A_{h}S_{h}\right\Vert \ =0$, then it is not
necessary that $\mathop{\lim }_{h\rightarrow 0}\left\Vert
T_{h}A_{h}-A_{h}T_{h}\right\Vert \ =0$.
\end{description}
\end{proposition}

\begin{proof}
i) We prove that%
\begin{equation*}
{\mathop{\lim }_{h\rightarrow 0}\left\Vert {\left( S_{h}-T_{h}\right) }^{%
\left[ n\right] }\right\Vert \ }={\mathop{\lim }_{h\rightarrow 0}\left\Vert {%
\left( T_{h}-S_{h}\right) }^{\left[ n\right] }\right\Vert \ }=0,
\end{equation*}%
for any $n\in {\mathbb{N}}$.

Since ${\left( T-S\right) }^{\left[ n+1\right] }=T{\left( T-S\right) }^{%
\left[ n\right] }-{\left( T-S\right) }^{\left[ n\right] }S$, for any $n\in {%
\mathbb{N}}$, taking $n=2$, it follows that%
\begin{equation*}
{\mathop{\lim \sup}_{h\rightarrow 0}\left\Vert {\left( T_{h}-S_{h}\right) }^{%
\left[ 2\right] }\right\Vert \ }={\mathop{\lim \sup}_{h\rightarrow 0}\left\Vert
T_{h}\left( T_{h}-S_{h}\right) -\left( T_{h}-S_{h}\right) S_{h}\right\Vert \ 
}\leq 
\end{equation*}%
\begin{equation*}
\leq {\mathop{\lim \sup}_{h\rightarrow 0}\left\Vert T_{h}\left(
T_{h}-S_{h}\right) \right\Vert \ }+{\mathop{\lim \sup}_{h\rightarrow
0}\left\Vert \left( T_{h}-S_{h}\right) S_{h}\right\Vert \ }\leq 
\end{equation*}%
\begin{equation*}
\leq {\mathop{\lim \sup}_{h\rightarrow 0}\left\Vert T_{h}\right\Vert \left\Vert
\left( T_{h}-S_{h}\right) \right\Vert \ }+{\mathop{\lim \sup}_{h\rightarrow
0}\left\Vert \left( T_{h}-S_{h}\right) \right\Vert \ }\left\Vert
S_{h}\right\Vert \leq 0.
\end{equation*}

By induction, we prove that if ${\mathop{\lim }_{h\rightarrow 0}\left\Vert {%
\left( T_{h}-S_{h}\right) }^{\left[ n\right] }\right\Vert \ }=0$, then ${%
\mathop{\lim }_{h\rightarrow 0}\left\Vert {\left( T_{h}-S_{h}\right) }^{%
\left[ n+1\right] }\right\Vert \ }=0$.%
\begin{equation*}
{\mathop{\lim \sup}_{h\rightarrow 0}\left\Vert {\left( T_{h}-S_{h}\right) }^{%
\left[ n+1\right] }\right\Vert \ }={\mathop{\lim \sup}_{h\rightarrow
0}\left\Vert T_{h}{\left( T_{h}-S_{h}\right) }^{\left[ n\right] }-{\left(
T_{h}-S_{h}\right) }^{\left[ n\right] }S_{h}\right\Vert \ }\leq 
\end{equation*}%
\begin{equation*}
\leq {\mathop{\lim \sup}_{h\rightarrow 0}\left\Vert T_{h}{\left(
T_{h}-S_{h}\right) }^{\left[ n\right] }\right\Vert \ }+{\mathop{\lim \sup}%
_{h\rightarrow 0}\left\Vert {\left( T_{h}-S_{h}\right) }^{\left[ n\right]
}S_{h}\right\Vert \ }\leq 
\end{equation*}%
\begin{equation*}
\leq {\mathop{\lim \sup}_{h\rightarrow 0}\left\Vert T_{h}\right\Vert \left\Vert {%
\left( T_{h}-S_{h}\right) }^{\left[ n\right] }\right\Vert \ }+{\mathop{\lim \sup}%
_{h\rightarrow 0}\left\Vert {\left( T_{h}-S_{h}\right) }^{\left[ n\right]
}\right\Vert \ }\left\Vert S_{h}\right\Vert \leq 0\text{.}
\end{equation*}

Similarly we can show that ${\mathop{\lim }_{h\rightarrow 0}\left\Vert {%
\left( S_{h}-T_{h}\right) }^{\left[ n\right] }\right\Vert \ }=0$, for any $%
n\in {\mathbb{N}}$.

When $n\rightarrow \infty $, we obtain%
\begin{equation*}
{\mathop{\lim }_{n\rightarrow \infty }{{\mathop{\lim \sup}_{h\rightarrow
0}\left\Vert {\left( S_{h}-T_{h}\right) }^{\left[ n\right] }\right\Vert \ }}%
^{\frac{1}{n}}\ }={\mathop{\lim }_{n\rightarrow \infty }{{\mathop{\lim \sup}%
_{h\rightarrow 0}\left\Vert {\left( T_{h}-S_{h}\right) }^{\left[ n\right]
}\right\Vert \ }}^{\frac{1}{n}}\ }=0\text{.}
\end{equation*}

ii) We remember that for any two bounded linear operators \textit{T }and%
\textit{\ S}, we have%
\begin{equation*}
{\left( T-S\right) }^{\left[ n\right] }=\sum_{k=0}^{n}{{\left( -1\right) }%
^{n-k}C_{k}^{n}T^{k}S^{n-k}}={\left( S-T\right) }^{\left[ n\right]
}+\sum_{k=0}^{n-1}{{\left( -1\right) }%
^{n-1-k}C_{k}^{n}(T^{k}S^{n-k}-S^{n-k}T^{k})},
\end{equation*}%
where $n\in {\mathbb{N}}$.

Applying above relation to $S_{h}$ \c{s}i $T_{h}$, when $h\rightarrow 0$, we
obtain%
\begin{equation*}
{\mathop{\lim \sup}_{h\rightarrow 0}\left\Vert {\left( T_{h}-S_{h}\right) }^{%
\left[ n\right] }\right\Vert \ }=
\end{equation*}%
\begin{equation*}
={\mathop{\lim \sup}_{h\rightarrow 0}\left\Vert {\left( S_{h}-T_{h}\right) }^{%
\left[ n\right] }-\sum_{k=0}^{n-1}{{\left( -1\right) }^{n-1-k}C_{k}^{n}({%
T_{h}}^{k}{S_{h}}^{n-k}-{S_{h}}^{n-k}{T_{h}}^{k})}\right\Vert \ }\leq 
\end{equation*}%
\begin{equation*}
\leq {\mathop{\lim \sup}_{h\rightarrow 0}\left\Vert {\left( S_{h}-T_{h}\right) }%
^{\left[ n\right] }\right\Vert \ }+{\mathop{\lim \sup}_{h\rightarrow
0}\left\Vert \sum_{k=0}^{n-1}{{\left( -1\right) }^{n-1-k}C_{k}^{n}({T_{h}}%
^{k}{S_{h}}^{n-k}-{S_{h}}^{n-k}{T_{h}}^{k})}\right\Vert \ }\leq 
\end{equation*}%
\begin{equation*}
\leq {\mathop{\lim \sup}_{h\rightarrow 0}\left\Vert {\left( S_{h}-T_{h}\right) }%
^{\left[ n\right] }\right\Vert \ }+\sum_{k=0}^{n-1}{C_{k}^{n}{\mathop{\lim \sup}%
_{h\rightarrow 0}\left\Vert {T_{h}}^{k}{S_{h}}^{n-k}-{S_{h}}^{n-k}{T_{h}}%
^{k}\right\Vert \ }}.
\end{equation*}

By Proposition 8 ii), it follows%
\begin{equation*}
{\mathop{\lim \sup}_{h\rightarrow 0}\left\Vert {\left( T_{h}-S_{h}\right) }^{%
\left[ n\right] }\right\Vert \ }\leq {\mathop{\lim \sup}_{h\rightarrow
0}\left\Vert {\left( S_{h}-T_{h}\right) }^{\left[ n\right] }\right\Vert }%
\text{,}
\end{equation*}

for any $n\in {\mathbb{N}}$.

Analogously we can prove that ${\mathop{\lim \sup}_{h\rightarrow 0}\left\Vert {%
\left( S_{h}-T_{h}\right) }^{\left[ n\right] }\right\Vert \ }\leq {%
\mathop{\lim \sup}_{h\rightarrow 0}\left\Vert {\left( T_{h}-S_{h}\right) }^{%
\left[ n\right] }\right\Vert }$.

iii) We suppose that the relation ${\mathop{\lim }_{h\rightarrow
0}\left\Vert T_{h}A_{h}-A_{h}T_{h}\right\Vert \ }=0$ is true. Then, taking $%
A_{h}=S_{h}$, for any $h\in \left. (0,1\right] $, since 
\begin{equation*}
{\mathop{\lim }_{h\rightarrow 0}\left\Vert {S_{h}}^{2}-{S_{h}}%
^{2}\right\Vert \ }=0\text{,}
\end{equation*}%
it follows 
\begin{equation*}
{\mathop{\lim }_{h\rightarrow 0}\left\Vert S_{h}T_{h}-T_{h}S_{h}\right\Vert
\ }=0\text{,}
\end{equation*}

fact that is not true.
\end{proof}

\noindent

\begin{proposition}
\label{p2.6}Let $\left\{ S_{h}\right\} ,\ \left\{ T_{h}\right\} \ \subset
L(X)$ be two asymptotically quasinilpotent equivalent families and $\left\{
A_{h}\right\} \ \subset L(X)$ a bounded family. Then

\begin{description}
\item[i)] The families $\left\{ S_{h}+A_{h}\right\} ,\ \left\{
T_{h}+A_{h}\right\} $ are asymptotically quasinilpotent equivalent;

\item[ii)] If $\left\{ A_{h}\right\} \ \subset L(X)$ is a bounded family
such that $\mathop{{\rm lim}}_{h\rightarrow 0}\left\Vert S_{h}A_{h}-{A_{h}S}%
_{h}\right\Vert =0$ and $\mathop{{\rm lim}}_{h\rightarrow 0}\left\Vert
T_{h}A_{h}-{A_{h}T}_{h}\right\Vert =0$ {, }the families $\left\{
S_{h}A_{h}\right\} ,\ \left\{ T_{h}A_{h}\right\} $ are asymptotically
quasinilpotent equivalent.
\end{description}
\end{proposition}

\begin{proof}
i) Since $\left\{ S_{h}\right\} ,\ \left\{ T_{h}\right\} $ are asymptotically
quasinilpotent equivalent, i.e.%
\begin{equation*}
{\mathop{\lim }_{n\rightarrow \infty }{{\mathop{\lim \sup}_{h\rightarrow
0}\left\Vert {\left( T_{h}-S_{h}\right) }^{\left[ n\right] }\right\Vert \ }}%
^{\frac{1}{n}}\ }={\mathop{\lim }_{n\rightarrow \infty }{{\mathop{\lim \sup}%
_{h\rightarrow 0}\left\Vert {\left( S_{h}-T_{h}\right) }^{\left[ n\right]
}\right\Vert \ }}^{\frac{1}{n}}\ }=0\text{,}
\end{equation*}%
then 
\begin{eqnarray*}
&&{\mathop{\lim }_{n\rightarrow \infty }{{\mathop{\lim \sup}_{h\rightarrow
0}\left\Vert {\left( \left( T_{h}+A_{h}\right) -(S_{h}+A_{h})\right) }^{%
\left[ n\right] }\right\Vert \ }}^{\frac{1}{n}}\ } \\
&=&{\mathop{\lim }_{n\rightarrow \infty }{{\mathop{\lim \sup}_{h\rightarrow
0}\left\Vert {\left( \left( S_{h}+A_{h}\right) -(T_{h}+A_{h})\right) }^{%
\left[ n\right] }\right\Vert \ }}^{\frac{1}{n}}\ } \\
&=&0\text{,}
\end{eqnarray*}

so $\left\{ S_{h}+A_{h}\right\} ,\ \left\{ T_{h}+A_{h}\right\} $ are
asymptotically quasinilpotent equivalent.

ii) Since $\mathop{{\rm lim}}_{h\rightarrow 0}\left\Vert S_{h}A_{h}-{A_{h}S}%
_{h}\right\Vert =0$ and $\mathop{{\rm lim}}_{h\rightarrow 0}\left\Vert
T_{h}A_{h}-{A_{h}T}_{h}\right\Vert =0$, taking into account Proposition 9, it follows%
\begin{equation*}
\mathop{{\rm \lim \sup}}_{h\rightarrow 0}\left\Vert {\left(
T_{h}A_{h}-S_{h}A_{h}\right) }^{\left[ n\right] }-{\left( T_{h}-S_{h}\right) 
}^{\left[ n\right] }{A_{h}}^{n}\right\Vert =
\end{equation*}%
\begin{equation*}
=\mathop{{\lim \sup}}_{h\rightarrow 0}\left\Vert \sum_{k=0}^{n}{{\left(
-1\right) }^{n-k}C_{n}^{k}({T_{h}A_{h})}^{k}{\left( S_{h}A_{h}\right) }^{n-k}%
}-\sum_{k=0}^{n}{{\left( -1\right) }^{n-k}C_{n}^{k}{T_{h}}^{k}{S_{h}}^{n-k}}{%
A_{h}}^{n}\right\Vert =
\end{equation*}%
\begin{equation*}
=\mathop{{\lim \sup}}_{h\rightarrow 0}\left\Vert \sum_{k=0}^{n}{{\left(
-1\right) }^{n-k}C_{n}^{k}(({T_{h}A_{h})}^{k}{\left( S_{h}A_{h}\right) }%
^{n-k}-{T_{h}}^{k}{S_{h}}^{n-k}{A_{h}}^{k}{A_{h}}^{n-k})}\right\Vert \leq 
\end{equation*}%
\begin{equation*}
\leq \sum_{k=0}^{n}{C_{n}^{k}\ \mathop{{\lim \sup}}_{h\rightarrow 0}\Vert
\left( T_{h}A_{h}\right) ^{k}{\left( S_{h}A_{h}\right) }^{n-k}-{T_{h}}^{k}{{{%
A_{h}}^{k}S}_{h}}^{n-k}{A_{h}}^{n-k}+}
\end{equation*}%
\begin{equation*}
+{T_{h}}^{k}{{{A_{h}}^{k}S}_{h}}^{n-k}{A_{h}}^{n-k}-{T_{h}}^{k}{S_{h}}^{n-k}{%
A_{h}}^{k}{A_{h}}^{n-k}\Vert \leq 
\end{equation*}%
\begin{equation*}
\leq \sum_{k=0}^{n}{C_{n}^{k}\ \mathop{{\lim \sup}}_{h\rightarrow 0}\left\Vert
({T_{h}A_{h})}^{k}{\left( S_{h}A_{h}\right) }^{n-k}-{T_{h}}^{k}{{{A_{h}}^{k}S%
}_{h}}^{n-k}{A_{h}}^{n-k}\right\Vert }+
\end{equation*}%
\begin{equation*}
+\sum_{k=0}^{n}{C_{n}^{k}\ \mathop{{\lim \sup}}_{h\rightarrow 0}\left\Vert {%
T_{h}}^{k}{{{A_{h}}^{k}S}_{h}}^{n-k}{A_{h}}^{n-k}-{T_{h}}^{k}{S_{h}}^{n-k}{%
A_{h}}^{k}{A_{h}}^{n-k}\right\Vert }\leq 
\end{equation*}%
\begin{equation*}
\leq \sum_{k=0}^{n}{C_{n}^{k}\ \mathop{{\lim \sup}}_{h\rightarrow 0}\Vert ({%
T_{h}A_{h}})^{k}{\left( S_{h}A_{h}\right) }^{n-k}-({T_{h}A_{h}})^{k}{S_{h}}%
^{n-k}{A_{h}}^{n-k}+}
\end{equation*}%
\begin{equation*}
+({T_{h}A_{h}})^{k}{S_{h}}^{n-k}{A_{h}}^{n-k}-{T_{h}}^{k}{A_{h}}^{k}{S_{h}}%
^{n-k}{A_{h}}^{n-k}\Vert +
\end{equation*}%
\begin{equation*}
+\sum_{k=0}^{n}{C_{n}^{k}\ \mathop{{\lim \sup}}_{h\rightarrow 0}\left\Vert {%
T_{h}}^{k}\right\Vert \left\Vert {{{A_{h}}^{k}S}_{h}}^{n-k}-{S_{h}}^{n-k}{%
A_{h}}^{k})\right\Vert \left\Vert {A_{h}}^{n-k}\right\Vert }\leq 
\end{equation*}%
\begin{equation*}
\leq \sum_{k=0}^{n}{C_{n}^{k}\ \mathop{{\lim \sup}}_{h\rightarrow 0}\left\Vert
({T_{h}A_{h})}^{k}{\left( S_{h}A_{h}\right) }^{n-k}-({T_{h}A_{h})}^{k}{S_{h}}%
^{n-k}{A_{h}}^{n-k}\right\Vert }+
\end{equation*}%
\begin{equation*}
+\sum_{k=0}^{n}{C_{n}^{k}\ \mathop{{\lim \sup}}_{h\rightarrow 0}\left\Vert ({%
T_{h}A_{h})}^{k}{S_{h}}^{n-k}{A_{h}}^{n-k}-{T_{h}}^{k}{A_{h}}^{k}{S_{h}}%
^{n-k}{A_{h}}^{n-k}\right\Vert }\leq 
\end{equation*}%
\begin{equation*}
\leq \sum_{k=0}^{n}{C_{n}^{k}\ \mathop{{\lim \sup}}_{h\rightarrow 0}\left\Vert
({T_{h}A_{h})}^{k}\right\Vert \left\Vert {\left( S_{h}A_{h}\right) }^{n-k}-{%
S_{h}}^{n-k}{A_{h}}^{n-k}\right\Vert }+
\end{equation*}%
\begin{equation*}
+\sum_{k=0}^{n}{C_{n}^{k}\ \mathop{{\lim \sup}}_{h\rightarrow 0}\left\Vert ({%
T_{h}A_{h})}^{k}-{T_{h}}^{k}{A_{h}}^{k}{S_{h}}^{n-k}{A_{h}}^{n-k}\right\Vert
\left\Vert {S_{h}}^{n-k}\right\Vert \left\Vert {A_{h}}^{n-k}\right\Vert }=0.
\end{equation*}

Having in view that $\left\{ S_{h}\right\} ,\ \left\{ T_{h}\right\} $ are
asymptotically quasinilpotent equivalent and taking into account the above
relation, it results%
\begin{eqnarray*}
&&{\mathop{\lim }_{n\rightarrow \infty }{{\mathop{\lim \sup}_{h\rightarrow
0}\left\Vert {\left( T_{h}A_{h}-S_{h}A_{h}\right) }^{\left[ n\right]
}\right\Vert \ }}^{\frac{1}{n}}} \\
&=&{\mathop{\lim }_{n\rightarrow \infty }{{\mathop{\lim \sup}_{h\rightarrow
0}\left\Vert {(T_{h}-S_{h})}^{\left[ n\right] }{A_{h}}^{n}\right\Vert \ }}^{%
\frac{1}{n}}} \\
&\leq &{\mathop{\lim }_{n\rightarrow \infty }{({\mathop{\lim \sup}_{h\rightarrow
0}\left\Vert {\left( T_{h}-S_{h}\right) }^{\left[ n\right] }\right\Vert
\left\Vert {A_{h}}^{n}\right\Vert \ })}^{\frac{1}{n}}\ } \\
&\leq &{\mathop{\lim }_{n\rightarrow \infty }{{\mathop{\lim }_{h\rightarrow
0}\left\Vert {\left( T_{h}-S_{h}\right) }^{\left[ n\right] }\right\Vert \ }}%
^{\frac{1}{n}}\ }{\mathop{\lim \sup}_{h\rightarrow 0}\left\Vert A_{h}\right\Vert
\ }=0.
\end{eqnarray*}

Analogously we can prove that ${\mathop{\lim }_{n\rightarrow \infty }{{%
\mathop{\lim \sup}_{h\rightarrow 0}\left\Vert {\left(
S_{h}A_{h}-T_{h}A_{h}\right) }^{\left[ n\right] }\right\Vert \ }}^{\frac{1}{n%
}}\ }=0$.
\end{proof}

\section{Spectrum of a family of operators}

\noindent Let be the sets

\noindent 
\[C_b\left(\left.(0,1\right],\ B\left(X\right)\right)=\left\{\left.\varphi :\left.(0,1\right]\to B\left(X\right)\right|\varphi \left(h\right)=T_h\ {\rm such\ that}\ \varphi \ is\ {\rm countinous\ and\ bounded}\right\}=\]
\[=\left\{\left.{\left\{T_h\right\}}_{h\in \left.(0,1\right]}\subset B(X)\right|{\left\{T_h\right\}}_{h\in \left.(0,1\right]}{\rm \ }{\rm is\ a\ bounded\ family}{\rm ,\ i.e.\ }{\mathop{sup}_{h\in \left.(0,1\right]} \left\|T_h\right\|\ }<\infty \right\}.\] 
and
\[C_0\left(\left.(0,1\right],\ B\left(X\right)\right)=\left\{\left.\varphi \in C_b\left(\left.(0,1\right],\ B\left(X\right)\right)\right|{\mathop{\lim }_{h\to 0} \left\|\varphi (h)\right\|=0\ }\right\}=\]
\[=\left\{\left.{\left\{T_h\right\}}_{h\in \left.(0,1\right]}\subset B(X)\right|{\mathop{\lim }_{h\to 0} \left\|T_h\right\|\ }=0\right\}.\]

\noindent $C_b\left(\left.(0,1\right],\ B\left(X\right)\right)\ $ is a Banach algebra non-commutative with  norm

\noindent 
\[\left\|\left\{T_h\right\}\right\|={sup}_{h\in \left.(0,1\right]}\left\|T_h\right\|,\]

\noindent and $C_0\left(\left.(0,1\right],\ B\left(X\right)\right)$ is a closed  bilateral ideal of $C_b\left(\left.(0,1\right],\ B\left(X\right)\right)$. Therefore the quotient algebra $C_b\left(\left.(0,1\right],\ B\left(X\right)\right)/C_0\left(\left.(0,1\right],\ B\left(X\right)\right)$, which will be called from now $B_{\infty }$, is also a Banach algebra with quotient norm

\noindent 
\[\left\|\dot{\left\{T_h\right\}}\right\|={inf}_{{\left\{U_h\right\}}_{h\in \left.(0,1\right]}\in C_0\left(\left.(0,1\right],\ B\left(X\right)\right)}\left\|\left\{T_h\right\}+\left\{U_h\right\}\right\|={inf}_{{\left\{S_h\right\}}_{h\in \left.(0,1\right]}\in \dot{\left\{T_h\right\}}}\left\|\left\{S_h\right\}\right\|.\] 
Then 
\[\left\|\dot{\left\{T_h\right\}}\right\|={inf}_{{\left\{S_h\right\}}_{h\in \left.(0,1\right]}\in \dot{\left\{T_h\right\}}}\left\|\left\{S_h\right\}\right\|\le \left\|\left\{S_h\right\}\right\|={sup}_{h\in \left.(0,1\right]}\left\|S_h\right\|,\]

\noindent for any  ${\left\{S_h\right\}}_{h\in \left.(0,1\right]}\in \dot{\left\{T_h\right\}}$. Moreover,

\noindent 
\[\left\|\dot{\left\{T_h\right\}}\right\|={inf}_{{\left\{S_h\right\}}_{h\in \left.(0,1\right]}\in \dot{\left\{T_h\right\}}}\left\|\left\{S_h\right\}\right\|={inf}_{{\left\{S_h\right\}}_{h\in \left.(0,1\right]}\in \dot{\left\{T_h\right\}}}{sup}_{h\in \left.(0,1\right]}\left\|S_h\right\|.\]

\noindent If two bounded families ${\left\{T_h\right\}}_{h\in \left.(0,1\right]},\ {\left\{S_h\right\}}_{h\in \left.(0,1\right]}\subset B(X)$ are asymptotically equivalent, then $\mathop{{\rm lim}}_{h\to 0}\left\|S_h-T_h\right\|=0$, i.e.  ${\left\{T_h-S_h\right\}}_{h\in \left.(0,1\right]}\in C_0\left(\left.(0,1\right],\ B\left(X\right)\right)$. 

\noindent Let ${\left\{T_h\right\}}_{h\in \left.(0,1\right]},\ {\left\{S_h\right\}}_{h\in \left.(0,1\right]}\in C_b\left(\left.(0,1\right],\ B\left(X\right)\right)$ be asymptotically equivalent. Then
\[{\mathop{{\lim \sup}}_{h\to 0} \left\|S_h\right\|\ }={\mathop{{\lim \sup}}_{h\to 0} \left\|T_h\right\|\ }.\] 
Since
\[{\mathop{{\lim \sup}}_{h\to 0} \left\|S_h\right\|\ }\le {\sup}_{h\in \left.(0,1\right]}\left\|S_h\right\|,\] 
results that

\noindent 
\[{\mathop{{\lim \sup}}_{h\to 0} \left\|S_h\right\|\ }={inf}_{{\left\{S_h\right\}}_{h\in \left.(0,1\right]}\in \dot{\left\{T_h\right\}}}{\mathop{{\lim \sup}}_{h\to 0} \left\|S_h\right\|\ }\le {inf}_{{\left\{S_h\right\}}_{h\in \left.(0,1\right]}\in \dot{\left\{T_h\right\}}}{\sup}_{h\in \left.(0,1\right]}\left\|S_h\right\|=\left\|\dot{\left\{T_h\right\}}\right\|,\]

\noindent for any ${\left\{S_h\right\}}_{h\in \left.(0,1\right]}\in \dot{\left\{T_h\right\}}$.

\noindent In particular
\[{\mathop{{\lim \lim}}_{h\to 0} \left\|T_h\right\|\ }\le \left\|\dot{\left\{T_h\right\}}\right\|\le \left\|\left\{T_h\right\}\right\|={\sup}_{h\in \left.(0,1\right]}\left\|T_h\right\|.\] 

\begin{definition}
\label{d3.1}We call the \textit{resolvent set }of a family of operators $%
\left\{ S_{h}\right\} \ \in C_b\left(\left.(0,1\right],\ B\left(X\right)\right)$ the set\noindent 
\begin{equation*}
r\left( \left\{ S_{h}\right\} \right) =\{\left. \lambda \in {\mathbb{C}}%
\right\vert \exists \left\{ {\mathcal{R}}(\lambda ,S_{h})\right\} \in C_b\left(\left.(0,1\right],\ B\left(X\right)\right),\ {\ \mathop{\lim}_{h\rightarrow 0}\left\Vert \left(
\lambda I-S_{h}\right) {\mathcal{R}}\left( \lambda ,S_{h}\right)
-I\right\Vert \ }=
\end{equation*}%
\begin{equation*}
={\mathop{lim}_{h\rightarrow 0}\left\Vert {\mathcal{R}}\left( \lambda
,S_{h}\right) \left( \lambda I-S_{h}\right) -I\right\Vert \ }=0\}
\end{equation*}

We call the \textit{spectrum }of a family of operators $\left\{
S_{h}\right\} \ \in C_b\left(\left.(0,1\right],\ B\left(X\right)\right) $ the set%
\begin{equation*}
Sp\left( \left\{ S_{h}\right\} \right) \mathrm{=}{\mathbb{C}}\backslash
r\left( \left\{ S_{h}\right\} \right) \text{.}
\end{equation*}
\end{definition}

\begin{remark}
\label{r3.2}

\begin{description}
\item[i)] If $\lambda \in r\left( S_{h}\right) $ for any $h\in (\left. 0,1%
\right] $, then $\lambda \in r\left( \left\{ S_{h}\right\} \right) .$
Therefore $\bigcap_{h\in (\left. 0,1\right] }{r\left( S_{h}\right) }%
\subseteq r\left( \left\{ S_{h}\right\} \right) $;

\item[ii)] If $\lambda \in Sp\left( \left\{ S_{h}\right\} \right) $, then $%
\left\vert \lambda \right\vert \leq {\mathop{\lim \sup }_{h\rightarrow 0}\left\Vert {S_{h}}\right\Vert \ 
}\ $;

\item[iii)] If $\left\Vert S_{h}\right\Vert <\left\vert \lambda \right\vert $
for any $h\in (\left. 0,1\right] $, then $\lambda \in r\left( \left\{
S_{h}\right\} \right) $;

\item[iv)]  $r\left( \left\{ S_{h}\right\} \right) $ is an open set of $C$
and $Sp\left( \left\{
S_{h}\right\} \right) $ is a compact set of $C$.
\end{description}
\end{remark}

\begin{proof}
iv) Let $\lambda \in \ r\left( \left\{ S_{h}\right\} \right) $. From
Definition 11, it follows that there is $\left\{ {\mathcal{R}}%
(\lambda ,S_{h})\right\} \in C_b\left(\left.(0,1\right],\ B\left(X\right)\right)$ such that%
\begin{equation*}
{\mathop{\lim }_{h\rightarrow 0}\left\Vert \left( \lambda I-S_{h}\right) {%
\mathcal{R}}\left( \lambda ,S_{h}\right) -I\right\Vert \ }={\mathop{\lim }%
_{h\rightarrow 0}\left\Vert {\mathcal{R}}\left( \lambda ,S_{h}\right) \left(
\lambda I-S_{h}\right) -I\right\Vert \ }=0.
\end{equation*}

Let $\mu \in \ D(\lambda ,\frac{1}{{\mathop{\lim \sup}_{h\rightarrow
0}\left\Vert {\mathcal{R}}\left( \lambda ,S_{h}\right) \right\Vert \ }})$.
So 
\begin{equation*}
\left\vert \lambda -\mu \right\vert <\frac{1}{{\mathop{\lim \sup}_{h\rightarrow
0}\left\Vert {\mathcal{R}}\left( \lambda ,S_{h}\right) \right\Vert \ }}.
\end{equation*}

According to ii), it follows $1\in r\left( \left\{ (\lambda -\mu ){\mathcal{R%
}}\left( \lambda ,S_{h}\right) \right\} \right) $, therefore there is {$%
\left\{ {\mathcal{R}}(1,\left( \lambda -\mu \right) {\mathcal{R}}\left(
\lambda ,S_{h}\right) )\right\} \in C_b\left(\left.(0,1\right],\ B\left(X\right)\right) $ such that }%
\begin{equation*}
{\mathop{\lim }_{h\rightarrow 0}\left\Vert \left( I-(\lambda -\mu ){\mathcal{%
R}}\left( \lambda ,S_{h}\right) \right) {\mathcal{R}}\left( 1,(\lambda -\mu )%
{\mathcal{R}}\left( \lambda ,S_{h}\right) \right) -I\right\Vert \ }=
\end{equation*}%
\begin{equation*}
={\mathop{\lim }_{h\rightarrow 0}\left\Vert {\mathcal{R}}\left( 1,(\lambda
-\mu ){\mathcal{R}}\left( \lambda ,S_{h}\right) \right) \left( I-(\lambda
-\mu ){\mathcal{R}}\left( \lambda ,S_{h}\right) \right) -I\right\Vert \ }=0.
\end{equation*}

Having in view the above relation, it results%
\begin{equation*}
{\mathop{\lim \sup}_{h\rightarrow 0}\left\Vert \left( \mu I-S_{h}\right) {%
\mathcal{R}}\left( \lambda ,S_{h}\right) {\mathcal{R}}\left( 1,\left(
\lambda -\mu \right) {\mathcal{R}}\left( \lambda ,S_{h}\right) \right)
-I\right\Vert \ }=
\end{equation*}%
\begin{equation*}
=\mathop{\lim \sup}_{h\rightarrow 0}\Vert \left( \lambda I-S_{h}\right) {%
\mathcal{R}}\left( \lambda ,S_{h}\right) {\mathcal{R}}\left( 1,\left(
\lambda -\mu \right) {\mathcal{R}}\left( \lambda ,S_{h}\right) \right) -
\end{equation*}%
\begin{equation*}
-\left( \lambda -\mu \right) {\mathcal{R}}\left( \lambda ,S_{h}\right) {%
\mathcal{R}}\left( 1,\left( \lambda -\mu \right) {\mathcal{R}}\left( \lambda
,S_{h}\right) \right) -I\Vert \ =
\end{equation*}%
\begin{equation*}
=\mathop{\lim \sup}_{h\rightarrow 0}\Vert \left( \left( \lambda I-S_{h}\right) {%
\mathcal{R}}\left( \lambda ,S_{h}\right) -I\right) {\mathcal{R}}\left(
1,\left( \lambda -\mu \right) {\mathcal{R}}\left( \lambda ,S_{h}\right)
\right) +{\mathcal{R}}\left( 1,\left( \lambda -\mu \right) {\mathcal{R}}%
\left( \lambda ,S_{h}\right) \right) -
\end{equation*}%
\begin{equation*}
-\left( \lambda -\mu \right) {\mathcal{R}}\left( \lambda ,S_{h}\right) {%
\mathcal{R}}\left( 1,(\lambda -\mu ){\mathcal{R}}\left( \lambda
,S_{h}\right) \right) -I\Vert \ \leq 
\end{equation*}%
\begin{equation*}
\leq {\mathop{\lim \sup}_{h\rightarrow 0}\left\Vert \left( \left( \lambda
I-S_{h}\right) {\mathcal{R}}\left( \lambda ,S_{h}\right) -I\right) {\mathcal{%
R}}\left( 1,\left( \lambda -\mu \right) {\mathcal{R}}\left( \lambda
,S_{h}\right) \right) \right\Vert \ }+
\end{equation*}%
\begin{equation*}
+\mathop{\lim \sup}_{h\rightarrow 0}\left\Vert {\mathcal{R}}\left( 1,\left(
\lambda -\mu \right) {\mathcal{R}}\left( \lambda ,S_{h}\right) \right)
-\left( \lambda -\mu \right) {\mathcal{R}}\left( \lambda ,S_{h}\right) {%
\mathcal{R}}\left( 1,(\lambda -\mu ){\mathcal{R}}\left( \lambda
,S_{h}\right) \right) -I\right\Vert \ \leq 
\end{equation*}%
\begin{equation*}
\leq {\mathop{\lim \sup}_{h\rightarrow 0}\left\Vert \left( \left( \lambda
I-S_{h}\right) {\mathcal{R}}\left( \lambda ,S_{h}\right) -I\right)
\right\Vert \left\Vert {\mathcal{R}}\left( 1,\left( \lambda -\mu \right) {%
\mathcal{R}}\left( \lambda ,S_{h}\right) \right) \right\Vert \ }+
\end{equation*}%
\begin{equation*}
+{\mathop{\lim }_{h\rightarrow 0}\left\Vert (I-\left( \lambda -\mu \right) {%
\mathcal{R}}\left( \lambda ,S_{h}\right) ){\mathcal{R}}\left( 1,(\lambda
-\mu ){\mathcal{R}}\left( \lambda ,S_{h}\right) \right) -I\right\Vert \ }=0,
\end{equation*}

so $\mu \in \ r\left( \left\{ S_{h}\right\} \right) $, for every $\mu \in \
D(\lambda ,\frac{1}{{\mathop{\lim \sup}_{h\rightarrow 0}\left\Vert {\mathcal{R}}%
\left( \lambda ,S_{h}\right) \right\Vert \ }})$. Therefore, for any $\lambda
\in \ r\left( \left\{ S_{h}\right\} \right) $, there is an open disk $%
D(\lambda ,\frac{1}{{\mathop{\lim \sup}_{h\rightarrow 0}\left\Vert {\mathcal{R}}%
\left( \lambda ,S_{h}\right) \right\Vert \ }})$ such that $D(\lambda ,\frac{1%
}{{\mathop{\lim \sup}_{h\rightarrow 0}\left\Vert {\mathcal{R}}\left( \lambda
,S_{h}\right) \right\Vert \ }})\subset \ r\left( \left\{ S_{h}\right\}
\right) $.

If $\left\{ S_{h}\right\} $ is a bounded family, from ii) we have%
\begin{equation*}
\left\vert \lambda \right\vert \leq {\mathop{\lim \sup }_{h\rightarrow
0}\left\Vert S_{h}\right\Vert \ }<\infty \text{,}
\end{equation*}

for any $\lambda \in \ Sp\left( \left\{ S_{h}\right\} \right) $, so $%
Sp\left( \left\{ S_{h}\right\} \right) $ is a compact set.
\end{proof}

\begin{proposition}
\label{d3.3} Let $\left\{ S_{h}\right\} \ \in C_b\left(\left.(0,1\right],\ B\left(X\right)\right) $ be
a family of operators and $\lambda \in r\left( \left\{ S_{h}\right\} \right) 
$. Then, for any $\left\{ {\mathcal{R}}\left( \lambda ,S_{h}\right) \right\}
\in C_b\left(\left.(0,1\right],\ B\left(X\right)\right) \ $ such that 
\[{\ \mathop{lim}_{h\rightarrow
0}\left\Vert \left( \lambda I-S_{h}\right) {\mathcal{R}}\left( \lambda
,S_{h}\right) -I\right\Vert \ }={\mathop{\lim}_{h\rightarrow 0}\left\Vert {%
\mathcal{R}}\left( \lambda ,S_{h}\right) \left( \lambda I-S_{h}\right)
-I\right\Vert \ }=0,\] we have 
\begin{equation*}
{\mathop{\lim}_{h\rightarrow 0}\left\Vert S_{h}{\mathcal{R}}\left( \lambda
,S_{h}\right) -{{\mathcal{R}}\left( \lambda ,S_{h}\right) S}_{h}\right\Vert
\ }=0.
\end{equation*}
\end{proposition}

\begin{proof}
\noindent Let $\lambda \in r\left(\left\{S_h\right\}\right)$ and $\left\{{%
\mathcal{R}}\left(\lambda ,S_h\right)\right\} \in C_b\left(\left.(0,1\right],\ B\left(X\right)\right) $ such
that

\noindent 
\begin{equation*}
{\ \mathop{\lim }_{h\to 0} \left\|\left(\lambda I-S_h\right){\mathcal{R}}%
\left(\lambda ,S_h\right)-I\right\|\ }={\mathop{\lim }_{h\to 0} \left\|{%
\mathcal{R}}\left(\lambda ,S_h\right)\left(\lambda I-S_h\right)-I\right\|\ }%
=0. 
\end{equation*}

\noindent Using this relation we have

\noindent 
\begin{equation*}
{\mathop{\lim \sup}_{h\to 0} \left\|S_h{\mathcal{R}}\left(\lambda ,S_h\right)-{{%
\mathcal{R}}\left(\lambda ,S_h\right)S}_h\right\|\ }={\mathop{\lim \sup}_{h\to
0} \left\|{\mathcal{R}}\left(\lambda ,S_h\right)\left(\lambda
I-S_h\right)-\left(\lambda I-S_h\right){\mathcal{R}}\left(\lambda
,S_h\right)\right\|\ }=
\end{equation*}
\begin{equation*}
={\mathop{\lim \sup}_{h\to 0} \left\|{\mathcal{R}}\left(\lambda
,S_h\right)\left(\lambda I-S_h\right)-I+I-\left(\lambda I-S_h\right){%
\mathcal{R}}\left(\lambda ,S_h\right)\right\|\ }\le
\end{equation*}
\begin{equation*}
\le {\mathop{\lim }_{h\to 0} \left\|\left(\lambda I-S_h\right){\mathcal{R}}%
\left(\lambda ,S_h\right)-I\right\|\ }+{\mathop{\lim }_{h\to 0} \left\|{%
\mathcal{R}}\left(\lambda ,S_h\right)\left(\lambda I-S_h\right)-I\right\|\ }%
=0.
\end{equation*}
\end{proof}

\begin{proposition} 
\label{d3.4} (resolvent equation - asymptotic) Let $%
\left\{S_h\right\}\ \in C_b\left(\left.(0,1\right],\ B\left(X\right)\right)$ be a bounded family and $\lambda ,\mu \in
r(\left\{S_h\right\})$. Then

\noindent 
\begin{equation*}
{\mathop{\lim}_{h\to 0} \left\|{\mathcal{R}}\left(\lambda ,S_h\right)-{%
\mathcal{R}}\left(\mu ,S_h\right)-\left(\mu -\lambda \right){\mathcal{R}}%
\left(\lambda ,S_h\right){\mathcal{R}}\left(\mu ,S_h\right)\right\|\ }=0.
\end{equation*}
\end{proposition}

\begin{proof}
\noindent Since $\left\{{\mathcal{R}}%
(\lambda ,S_h)\right\}$ and $\left\{{\mathcal{R}}(\lambda ,S_h)\right\}$ are
bounded, we have

\noindent 
\begin{equation*}
{\mathop{\lim \sup}_{h\to 0} \left\|{\mathcal{R}}\left(\lambda ,S_h\right)-{%
\mathcal{R}}\left(\mu ,S_h\right)-\left(\mu -\lambda \right){\mathcal{R}}%
\left(\lambda ,S_h\right)R\left(\mu ,S_h\right)\right\|\ }=
\end{equation*}
\begin{equation*}
={\mathop{\lim \sup}_{h\to 0} \left\|{\mathcal{R}}\left(\lambda
,S_h\right)(I-\mu {\mathcal{R}}\left(\mu ,S_h\right))-\left(I-\lambda {%
\mathcal{R}}\left(\lambda ,S_h\right)\right){\mathcal{R}}\left(\mu
,S_h\right)\right\|\ }=
\end{equation*}
\begin{equation*}
={\mathop{\lim \sup}_{h\to 0} \left\|{\mathcal{R}}\left(\lambda
,S_h\right)(I-(\mu I-S_h){\mathcal{R}}\left(\mu ,S_h\right))-\left(I-{%
\mathcal{R}}\left(\lambda ,S_h\right)(\lambda I-S_h)\right){\mathcal{R}}%
\left(\mu ,S_h\right)\right\|\ }\le
\end{equation*}
\begin{equation*}
\le {\mathop{\lim \sup}_{h\to 0} \left\|{\mathcal{R}}\left(\lambda
,S_h\right)(I-(\mu I-S_h){\mathcal{R}}\left(\mu ,S_h\right))\right\|+\ }{%
\mathop{\lim \sup}_{h\to 0} \left\|\left(I-{\mathcal{R}}\left(\lambda
,S_h\right)(\lambda I-S_h)\right){\mathcal{R}}\left(\mu ,S_h\right)\right\|\ 
}\le
\end{equation*}
\begin{equation*}
\le {\mathop{\lim \sup}_{h\to 0} \left\|{\mathcal{R}}\left(\lambda
,S_h\right)\right\|\left\|I-(\mu I-S_h){\mathcal{R}}\left(\mu
,S_h\right)\right\|+}
\end{equation*}
\begin{equation*}
+{\mathop{\lim \sup}_{h\to 0} \left\|I-{\mathcal{R}}\left(\lambda
,S_h\right)(\lambda I-S_h)\right\|\left\|{\mathcal{R}}\left(\mu
,S_h\right)\right\|\ }\le 0.
\end{equation*}
\end{proof}

\begin{corollary}
\label{d3.5} Let $\left\{S_h\right\}\in C_b\left(\left.(0,1\right],\ B\left(X\right)\right)$ be a bounded
family and $\lambda ,\mu \in r(\left\{S_h\right\})$ be not-equal. Then

\noindent 
\begin{equation*}
{\mathop{\lim}_{h\to 0} \left\|{\mathcal{R}}\left(\lambda ,S_h\right){%
\mathcal{R}}\left(\mu ,S_h\right)-{\mathcal{R}}\left(\mu ,S_h\right){%
\mathcal{R}}\left(\lambda ,S_h\right)\right\|\ }=0.
\end{equation*}
\end{corollary}

\begin{proof}
\noindent 
\begin{equation*}
{\mathop{\lim \sup}_{h\to 0} \left\|{\mathcal{R}}\left(\lambda ,S_h\right){%
\mathcal{R}}\left(\mu ,S_h\right)-{\mathcal{R}}\left(\mu ,S_h\right){%
\mathcal{R}}\left(\lambda ,S_h\right)\right\|\ }=
\end{equation*}
\begin{equation*}
=\frac{1}{\left|\lambda -\mu \right|}{\mathop{\lim \sup}_{h\to 0}
\left\|\left(\lambda -\mu \right){\mathcal{R}}\left(\lambda ,S_h\right){%
\mathcal{R}}\left(\mu ,S_h\right)+(\mu -\lambda ){\mathcal{R}}\left(\mu
,S_h\right){\mathcal{R}}\left(\lambda ,S_h\right)\right\|\ }=
\end{equation*}
\begin{equation*}
=\frac{1}{\left|\lambda -\mu \right|}{\mathop{\lim \sup}_{h\to 0} \|\left[{%
\mathcal{R}}\left(\lambda ,S_h\right)-{\mathcal{R}}\left(\mu
,S_h\right)-\left(\mu -\lambda \right){\mathcal{R}}\left(\lambda ,S_h\right){%
\mathcal{R}}\left(\mu ,S_h\right)\right]+}
\end{equation*}
\begin{equation*}
{+\left[{\mathcal{R}}\left(\mu ,S_h\right)-{\mathcal{R}}\left(\lambda
,S_h\right)-(\lambda -\mu ){\mathcal{R}}\left(\mu ,S_h\right){\mathcal{R}}%
\left(\lambda ,S_h\right)\right]\|\ }\le
\end{equation*}
\begin{equation*}
\le \frac{1}{\left|\lambda -\mu \right|}{\mathop{\lim }_{h\to 0} \left\|%
\left[{\mathcal{R}}\left(\lambda ,S_h\right)-{\mathcal{R}}\left(\mu
,S_h\right)-\left(\mu -\lambda \right){\mathcal{R}}\left(\lambda ,S_h\right){%
\mathcal{R}}\left(\mu ,S_h\right)\right]\right\|\ }+
\end{equation*}
\begin{equation*}
+\frac{1}{\left|\lambda -\mu \right|}{\mathop{\lim }_{h\to 0} \left\|\left[{%
\mathcal{R}}\left(\mu ,S_h\right)-{\mathcal{R}}\left(\lambda
,S_h\right)-(\lambda -\mu ){\mathcal{R}}\left(\mu ,S_h\right){\mathcal{R}}%
\left(\lambda ,S_h\right)\right]\right\|\ }=0.
\end{equation*}
\end{proof}

\begin{proposition}
\label{d3.6} Let $\left\{S_h\right\}\ \in C_b\left(\left.(0,1\right],\ B\left(X\right)\right)$ be a
bounded family. If $\lambda \in r\left(\left\{S_h\right\}\right)$ and $%
\left\{{{\mathcal{R}}}_i\left(\lambda ,S_h\right)\right\}\in C_b\left(\left.(0,1\right],\ B\left(X\right)\right),\ i=\overline{1,2}$ such that

\noindent 
\begin{equation*}
{\ \mathop{\lim}_{h\to 0} \left\|\left(\lambda I-S_h\right){{\mathcal{R}}}%
_i\left(\lambda ,S_h\right)-I\right\|\ }={\mathop{\lim}_{h\to 0} \left\|{{%
\mathcal{R}}}_i\left(\lambda ,S_h\right)\left(\lambda
I-S_h\right)-I\right\|\ }=0
\end{equation*}

\noindent for $i=\overline{1,2}$, then 
\begin{equation*}
{\mathop{\lim}_{h\to 0} \left\|{{\mathcal{R}}}_1\left(\lambda ,S_h\right)-{{%
\mathcal{R}}}_2\left(\lambda ,S_h\right)\right\|\ }=0.
\end{equation*}
\end{proposition}

\begin{proof}
\noindent Let $\lambda \in r\left(\left\{S_h\right\}\right)$ and $\left\{{{%
\mathcal{R}}}_i\left(\lambda ,S_h\right)\right\}\in C_b\left(\left.(0,1\right],\ B\left(X\right)\right),\ i=%
\overline{1,2},$ such that

\noindent 
\begin{equation*}
{\ \mathop{\lim }_{h\to 0} \left\|\left(\lambda I-S_h\right){{\mathcal{R}}}%
_i\left(\lambda ,S_h\right)-I\right\|\ }={\mathop{\lim }_{h\to 0} \left\|{{%
\mathcal{R}}}_i\left(\lambda ,S_h\right)\left(\lambda
I-S_h\right)-I\right\|\ }=0
\end{equation*}
Therefore

\noindent 
\begin{equation*}
{\mathop{\lim \sup}_{h\to 0} \left\|{{\mathcal{R}}}_1\left(\lambda ,S_h\right)-{{%
\mathcal{R}}}_2\left(\lambda ,S_h\right)\right\|\ }=
\end{equation*}
\begin{equation*}
={\mathop{\lim \sup}_{h\to 0} \left\|{{\mathcal{R}}}_1\left(\lambda ,S_h\right)-{%
{\mathcal{R}}}_2\left(\lambda ,S_h\right)-\left(\lambda -\lambda \right){{%
\mathcal{R}}}_1\left(\lambda ,S_h\right){{\mathcal{R}}}_2\left(\lambda
,S_h\right)\right\|\ }=
\end{equation*}
\begin{equation*}
={\mathop{\lim \sup}_{h\to 0} \left\|{{\mathcal{R}}}_1\left(\lambda
,S_h\right)\left(I-\lambda {{\mathcal{R}}}_2\left(\lambda
,S_h\right)\right)-\left(I-\lambda {{\mathcal{R}}}_1\left(\lambda
,S_h\right)\right){{\mathcal{R}}}_2\left(\lambda ,S_h\right)\right\|\ }=
\end{equation*}
\begin{equation*}
={\mathop{\lim \sup}_{h\to 0} \left\|{{\mathcal{R}}}_1\left(\lambda
,S_h\right)\left(I-\left(\lambda I-S_h\right){{\mathcal{R}}}_2\left(\lambda
,S_h\right)\right)-\left(I-{{\mathcal{R}}}_1\left(\lambda
,S_h\right)\left(\lambda I-S_h\right)\right){{\mathcal{R}}}_2\left(\lambda
,S_h\right)\right\|\ }\le
\end{equation*}
\begin{equation*}
\le {\mathop{\lim \sup}_{h\to 0} \left\|{{\mathcal{R}}}_1\left(\lambda
,S_h\right)\left(I-\left(\lambda I-S_h\right){{\mathcal{R}}}_2\left(\lambda
,S_h\right)\right)\right\|+\ }
\end{equation*}
\begin{equation*}
+{\mathop{\lim \sup}_{h\to 0} \left\|\left(I-{{\mathcal{R}}}_1\left(\lambda
,S_h\right)\left(\lambda I-S_h\right)\right){{\mathcal{R}}}_2\left(\lambda
,S_h\right)\right\|\ }\le
\end{equation*}
\begin{equation*}
\le {\mathop{\lim \sup}_{h\to 0} \left\|{{\mathcal{R}}}_1\left(\lambda
,S_h\right)\right\|\left\|I-\left(\lambda I-S_h\right){{\mathcal{R}}}%
_2\left(\lambda ,S_h\right)\right\|+\ }
\end{equation*}
\begin{equation*}
+{\mathop{\lim \sup}_{h\to 0} \left\|I-{{\mathcal{R}}}_1\left(\lambda
,S_h\right)(\lambda I-S_h)\right\|\left\|{{\mathcal{R}}}_2\left(\lambda
,S_h\right)\right\|\ }\le 0
\end{equation*}
\end{proof}

\begin{proposition} 
\label{d3.7} Let $\left\{S_h\right\}\ \in C_b\left(\left.(0,1\right],\ B\left(X\right)\right)$
be a bounded family, $\lambda \in r\left(\left\{S_h\right\}\right)$\ and $%
\left\{{\mathcal{R}}(\lambda ,S_h)\right\}\in C_b\left(\left.(0,1\right],\ B\left(X\right)\right)$ such that
\begin{equation*}
{\ \mathop{lim}_{h\to 0} \left\|\left(\lambda I-S_h\right){\mathcal{R}}%
\left(\lambda ,S_h\right)-I\right\|\ }={\mathop{lim}_{h\to 0} \left\|{%
\mathcal{R}}\left(\lambda ,S_h\right)\left(\lambda I-S_h\right)-I\right\|\ }%
=0.
\end{equation*}

If $\left\{R_h\right\}\ \in C_b\left(\left.(0,1\right],\ B\left(X\right)\right)$ is a bounded family such that it is
asymptotically equivalent with $\left\{{\mathcal{R}}(\lambda
,S_h)\right\}\in C_b\left(\left.(0,1\right],\ B\left(X\right)\right)$, then 
\begin{equation*}
{\mathop{\lim}_{h\to 0} \left\|\left(\lambda I-S_h\right)R_h-I\right\|\ }={%
\mathop{\lim}_{h\to 0} \left\|R_h\left(\lambda I-S_h\right)-I\right\|\ }=0.
\end{equation*}
\end{proposition}

\begin{proof}
\noindent Let $\lambda \in r\left(\left\{S_h\right\}\right)$. It results

\noindent 
\begin{equation*}
{\ \mathop{\lim \sup}_{h\to 0} \left\|\left(\lambda I-S_h\right)R_h-I\right\|\ }=
\end{equation*}
\begin{equation*}
={\ \mathop{\lim \sup}_{h\to 0} \left\|\left(\lambda
I-S_h\right)R_h-\left(\lambda I-S_h\right){\mathcal{R}}\left(\lambda
,S_h\right)+\left(\lambda I-S_h\right){\mathcal{R}}\left(\lambda
,S_h\right)-I\right\|\ }\le
\end{equation*}
\begin{equation*}
\le {\ \mathop{\lim \sup}_{h\to 0} \left\|\left(\lambda
I-S_h\right)R_h-\left(\lambda I-S_h\right){\mathcal{R}}\left(\lambda
,S_h\right)\right\|+\ }{\ \mathop{\lim \sup}_{h\to 0} \left\|\left(\lambda
I-S_h\right){\mathcal{R}}\left(\lambda ,S_h\right)-I\right\|\ }\le
\end{equation*}
\begin{equation*}
\le {\ \mathop{\lim \sup}_{h\to 0} \left\|\lambda I-S_h\right\|\left\|R_h-{%
\mathcal{R}}\left(\lambda ,S_h\right)\right\|\ }\le 0.
\end{equation*}

\noindent Analogously we can prove that ${\mathop{\lim }_{h\to 0}
\left\|R_h\left(\lambda I-S_h\right)-I\right\|\ }=0$.
\end{proof}

\begin{proposition}
\label{d3.8} Let\textit{ }$\left\{S_h\right\}\ \in C_b\left(\left.(0,1\right],\ B\left(X\right)\right)$\textit{, }$\lambda \in r\left(\left\{S_h\right\}\right)$\textit{ }and\textit{ }$\left\{{\mathcal R}(\lambda ,S_h)\right\}\in C_b\left(\left.(0,1\right],\ B\left(X\right)\right)$\textit{ }such that\textit{}
\[{\mathop{\lim}_{h\to 0} \left\|\left(\lambda I-S_h\right){\mathcal R}\left(\lambda ,S_h\right)-I\right\|\ }={\mathop{\lim}_{h\to 0} \left\|{\mathcal R}\left(\lambda ,S_h\right)\left(\lambda I-S_h\right)-I\right\|\ }=0.\] 
Then
\[{\ \mathop{\lim \sup}_{h\to 0} \left\|{\mathcal R}\left(\lambda ,S_h\right)\right\|\ne 0\ }.\] 
\end{proposition}

\begin{proof}
Suppose that ${\ \mathop{\lim \sup}_{h\to 0} \left\|{\mathcal R}\left(\lambda ,S_h\right)\right\|\ }=0$. Since 

\noindent 
\[{\rm 1=}\left\|I\right\|\le \left\|\left(\lambda I-S_h\right){\mathcal R}\left(\lambda ,S_h\right)-I\right\|+\left\|\left(\lambda I-S_h\right){\mathcal R}\left(\lambda ,S_h\right)\right\|\]

\noindent And taking into account that $\left\{S_h\right\}\ \in C_b\left(\left.(0,1\right],\ B\left(X\right)\right)$, it follows that

\noindent 
\[{{\rm 1 \le}\mathop{\lim \sup}_{h\to 0} \left\|\left(\lambda I-S_h\right){\mathcal R}\left(\lambda ,S_h\right)-I\right\|\ }+{\mathop{\lim \sup}_{h\to 0} \left\|\left(\lambda I-S_h\right){\mathcal R}\left(\lambda ,S_h\right)\right\|\ }\le\]
\[\le {\ \mathop{\lim \sup}_{h\to 0} \left\|\lambda I-S_h\right\|\left\|{\mathcal R}\left(\lambda ,S_h\right)\right\|\ }\le \left(\left|\lambda \right|+{\ \mathop{\lim \sup}_{h\to 0} \left\|S_h\right\|\ }\right){\ \mathop{{\lim \sup}}_{h\to 0} \left\|{\mathcal R}\left(\lambda ,S_h\right)\right\|\ }=0,\] 
contradiction.
\end{proof}
\noindent 

\noindent

\begin{proposition}
\label{d3.9} Let $\left\{S_h\right\}\ \in C_b\left(\left.(0,1\right],\ B\left(X\right)\right)$. If $\lambda ,\mu \in r(\left\{S_h\right\})$ such that there
are $\left\{{\mathcal{R}}(\lambda ,S_h)\right\},\ \ \left\{{\mathcal{R}}(\mu
,S_h)\right\}\in C_b\left(\left.(0,1\right],\ B\left(X\right)\right)$ with property ${\mathop{\lim}_{h\to 0}
\left\|{\mathcal{R}}\left(\lambda ,S_h\right)-{\mathcal{R}}\left(\mu
,S_h\right)\right\|\ }=0$, then $\lambda =\mu $.
\end{proposition}
\noindent

\begin{proof}
\noindent For $\lambda \in r\left(\left\{S_h\right\}\right)$ let $\left\{{%
\mathcal{R}}(\lambda ,S_h)\right\}\in C_b\left(\left.(0,1\right],\ B\left(X\right)\right)$ such that

\noindent 
\begin{equation*}
{\ \mathop{\lim }_{h\to 0} \left\|\left(\lambda I-S_h\right){\mathcal{R}}%
\left(\lambda ,S_h\right)-I\right\|\ }={\mathop{\lim }_{h\to 0} \left\|{%
\mathcal{R}}\left(\lambda ,S_h\right)\left(\lambda I-S_h\right)-I\right\|\ }%
=0
\end{equation*}

\noindent and for $\mu $$\in r\left(\left\{S_h\right\}\right)$ let $\left\{{%
\mathcal{R}}(\mu ,S_h)\right\}\in C_b\left(\left.(0,1\right],\ B\left(X\right)\right)$ such that

\noindent 
\begin{equation*}
{\ \mathop{\lim }_{h\to 0} \left\|\left(\mu I-S_h\right){\mathcal{R}}%
\left(\mu ,S_h\right)-I\right\|\ }={\mathop{\lim }_{h\to 0} \left\|{\mathcal{%
R}}\left(\mu ,S_h\right)\left(\mu I-S_h\right)-I\right\|\ }=0.
\end{equation*}
If 
\begin{equation*}
{\mathop{\lim }_{h\to 0} \left\|{\mathcal{R}}\left(\lambda ,S_h\right)-{%
\mathcal{R}}\left(\mu ,S_h\right)\right\|\ }=0,
\end{equation*}

\noindent Having in view Proposition \ref{d3.7}, we obtain

\noindent 
\begin{equation*}
{\ \mathop{\lim }_{h\to 0} \left\|\left(\mu I-S_h\right){\mathcal{R}}%
\left(\lambda ,S_h\right)-I\right\|\ }={\mathop{\lim }_{h\to 0} \left\|{%
\mathcal{R}}\left(\lambda ,S_h\right)\left(\mu I-S_h\right)-I\right\|\ }=0.
\end{equation*}
Hence

\noindent 
\begin{equation*}
{\ \mathop{\lim \sup}_{h\to 0} \left\|\left(\lambda I-S_h\right){\mathcal{R}}%
\left(\lambda ,S_h\right)-\left(\mu I-S_h\right){\mathcal{R}}\left(\lambda
,S_h\right)\right\|\ }\le
\end{equation*}
\begin{equation*}
\le {\ \mathop{\lim \sup}_{h\to 0} \left\|\left(\lambda I-S_h\right){\mathcal{R}}%
\left(\lambda ,S_h\right)-I\right\|\ }+{\ \mathop{\lim }_{h\to 0}
\left\|\left(\mu I-S_h\right){\mathcal{R}}\left(\lambda
,S_h\right)-I\right\|\ }=0.
\end{equation*}
Therefore 
\begin{equation*}
\left|\lambda -\mu \right|{\ \mathop{\lim \sup}_{h\to 0} \left\|{\mathcal{R}}%
\left(\lambda ,S_h\right)\right\|\ }=0
\end{equation*}

\noindent And according to Proposition \ref{d3.8} (${\ \mathop{\lim \sup}_{h\to 0}
\left\|R\left(\lambda ,S_h\right)\right\|\ne 0\ }$) it follows $\lambda =\mu 
$.
\end{proof}

\noindent

\begin{lemma}
\label{d3.10} If two bounded families $\left\{S_h\right\},\
\left\{T_h\right\}\ \in C_b\left(\left.(0,1\right],\ B\left(X\right)\right)$ are asymptotically equivalent and
there is $\left\{R_h\left(\lambda \right)\right\}\in C_b\left(\left.(0,1\right],\ B\left(X\right)\right)\ $
such that 
\begin{equation*}
{\mathop{\ lim}_{h\to 0} \left\|\left(\lambda I-S_h\right)R_h\left(\lambda
\right)-I\right\|\ }={\mathop{lim}_{h\to 0} \left\|R_h\left(\lambda
\right)\left(\lambda I-S_h\right)-I\right\|\ }=0,
\end{equation*}
then

\noindent 
\begin{equation*}
{\mathop{lim}_{h\to 0} \left\|\left(\lambda I-T_h\right)R_h\left(\lambda
\right)-I\right\|\ }={\mathop{lim}_{h\to 0} \left\|R_h\left(\lambda
\right)\left(\lambda I-T_h\right)-I\right\|\ }=0.
\end{equation*}
\end{lemma}

\begin{proof}
\noindent Since the two families $\left\{S_h\right\},\ \left\{T_h\right\}\
\in C_b\left(\left.(0,1\right],\ B\left(X\right)\right)$ are asymptotically equivalent, i.e. ${\mathop{\lim }_{h\to 0}
\left\|S_h-T_h\right\|=0\ }$, we have

\noindent 
\begin{equation*}
{\mathop{\lim \sup}_{h\to 0} \left\|\left(\lambda I-T_h\right)R_h\left(\lambda
\right)-I\right\|\ }=
\end{equation*}
\begin{equation*}
={\mathop{\lim \sup}_{h\to 0} \left\|\left(\lambda I-T_h\right)R_h\left(\lambda
\right)-\left(\lambda I-S_h\right)R_h\left(\lambda \right)+\left(\lambda
I-S_h\right)R_h\left(\lambda \right)-I\right\|\ }\le
\end{equation*}
\begin{equation*}
\le {\mathop{\lim \sup}_{h\to 0} \left\|\left(\lambda
I-T_h\right)R_h\left(\lambda \right)-\left(\lambda
I-S_h\right)R_h\left(\lambda \right)\right\|\ }+{\mathop{\lim \sup}_{h\to 0}
\left\|\left(\lambda I-S_h\right)R_h\left(\lambda \right)-I\right\|\ }=
\end{equation*}
\begin{equation*}
={\mathop{\lim \sup}_{h\to 0} \left\|T_hR_h\left(\lambda
\right)-S_hR_h\left(\lambda \right)\right\|\ }\le {\mathop{\lim \sup}_{h\to 0}
\left\|T_h-S_h\right\|\ }\left\|R_h\left(\lambda \right)\right\|\le 0.
\end{equation*}
\end{proof}

\noindent
\begin{remark}
\label{d3.11} Since $B_{\infty }=C_b\left(\left.(0,1\right],\ B\left(X\right)\right)/C_0\left(\left.(0,1\right],\ B\left(X\right)\right)$ is a Banach algebra, then make sense 

\noindent 
\[r\left(\dot{\left\{S_h\right\}}\right)=\left\{\left.\lambda \in {\mathbb C}\right|\exists \dot{\left\{R_h\right\}}\in \ B_{\infty }\ \ a.î.\ \left(\lambda \dot{\left\{I\right\}}-\dot{\left\{S_h\right\}}\right)\dot{\left\{R_h\right\}}=\dot{\left\{I\right\}}=\dot{\left\{R_h\right\}}\left(\lambda \dot{\left\{I\right\}}-\dot{\left\{S_h\right\}}\right)\right\}\] 
and
\[Sp\left(\dot{\left\{S_h\right\}}\right){\rm =}{\mathbb C}{\rm \backslash }{\rm \ }r\left(\dot{\left\{S_h\right\}}\right).\] 

\noindent Let $\left\{S_h\right\}\ \in C_b\left(\left.(0,1\right],\ B\left(X\right)\right)$ and $\lambda \in r(\left\{S_h\right\})$\textit{.} Fie $\left\{{\mathcal R}(\lambda ,S_h)\right\}\in C_b\left(\left.(0,1\right],\ B\left(X\right)\right)$ such that

\noindent 
\[{\ \mathop{\lim }_{h\to 0} \left\|\left(\lambda I-S_h\right){\mathcal R}\left(\lambda ,S_h\right)-I\right\|\ }={\mathop{\lim }_{h\to 0} \left\|{\mathcal R}\left(\lambda ,S_h\right)\left(\lambda I-S_h\right)-I\right\|\ }=0.\]

\noindent By Proposition \ref{d3.7}, it results that for any $\left\{{\mathcal R}'(\lambda ,S_h)\right\}\in \dot{\left\{{\mathcal R}(\lambda ,S_h)\right\}}$, we have 

\noindent 
\[{\ \mathop{\lim }_{h\to 0} \left\|\left(\lambda I-S_h\right){\mathcal R}'\left(\lambda ,S_h\right)-I\right\|\ }={\mathop{\lim }_{h\to 0} \left\|{\mathcal R}'\left(\lambda ,S_h\right)\left(\lambda I-S_h\right)-I\right\|\ }=0.\]

\noindent Moreover, for every  $\left\{{S'}_h\right\}\in \dot{\left\{S_h\right\}}$, by Lemma \ref{d3.10} we have

\noindent 
\[{\ \mathop{\lim }_{h\to 0} \left\|\left(\lambda I-{S'}_h\right){\mathcal R}\left(\lambda ,S_h\right)-I\right\|\ }={\mathop{\lim }_{h\to 0} \left\|{\mathcal R}\left(\lambda ,S_h\right)\left(\lambda I-{S'}_h\right)-I\right\|\ }=0.\]

\noindent Therefore every representative of class $\dot{\left\{{\mathcal R}(\lambda ,S_h)\right\}}\in \ B_{\infty }$ is an ''inverse'' for any representative of class $\dot{\left\{S_h\right\}}$. 
\end{remark}
\noindent 

\begin{theorem}
\label{d3.12} Let $\left\{S_h\right\}\in C_b\left(\left.(0,1\right],\ B\left(X\right)\right)$. Then

\noindent 
\[Sp\left(\dot{\left\{S_h\right\}}\right)=Sp\left(\left\{S_h\right\}\right).\] 
\end{theorem}

\begin{proof}
 Let $\lambda \in r\left(\dot{\left\{S_h\right\}}\right)$. Then there is $\dot{\left\{R_h\right\}}\in \ B_{\infty }\ \ $such that

\noindent 
\[\left(\lambda \dot{\left\{I\right\}}-\dot{\left\{S_h\right\}}\right)\dot{\left\{R_h\right\}}=\dot{\left\{I\right\}}=\dot{\left\{R_h\right\}}\left(\lambda \dot{\left\{I\right\}}-\dot{\left\{S_h\right\}}\right).\]

\noindent Taking into account the algebraic relations of the Banach algebra $B_{\infty }$, it results

\noindent 
\[\dot{\left\{I\right\}}=\left(\lambda \dot{\left\{I\right\}}-\dot{\left\{S_h\right\}}\right)\dot{\left\{R_h\right\}}=\ \dot{\left\{\lambda I-S_h\right\}}\dot{\left\{R_h\right\}}=\dot{\left\{(\lambda I-S_h)R_h\right\}}.\]

\noindent Therefore  $\left\{\left(\lambda I-S_h\right)R_h-I\right\}\in B_{\infty }$, i.e. 

\noindent 
\[{\mathop{{\rm l}{\rm im}}_{h\to 0} \left\|\left(\lambda I-S_h\right)R_h-I\right\|\ }=0.\]

\noindent Analogously we can show that ${\mathop{\lim }_{h\to 0} \left\|R_h\left(\lambda I-S_h\right)-I\right\|\ }=0$. Then  $\lambda \in r\left(\left\{S_h\right\}\right)$.

\noindent Conversely, let $\lambda \in r\left(\left\{S_h\right\}\right)$. Then there is $\left\{{\mathcal R}(\lambda ,S_h)\right\}\in C_b\left(\left.(0,1\right],\ B\left(X\right)\right)$ such that

\noindent 
\[{\ \mathop{\lim }_{h\to 0} \left\|\left(\lambda I-S_h\right){\mathcal R}\left(\lambda ,S_h\right)-I\right\|\ }={\mathop{\lim }_{h\to 0} \left\|{\mathcal R}\left(\lambda ,S_h\right)\left(\lambda I-S_h\right)-I\right\|\ }=0.\]

\noindent Let $\left\{R_h\right\}\in \dot{\left\{{\mathcal R}(\lambda ,S_h)\right\}}$. Then 

\noindent 
\[{\ \mathop{\lim }_{h\to 0} \left\|\left(\lambda I-S_h\right)R_h-I\right\|\ }={\mathop{\lim }_{h\to 0} \left\|R_h\left(\lambda I-S_h\right)-I\right\|\ }=0\]

\noindent and $\left\{\left(\lambda I-S_h\right)R_h-I\right\},\ \left\{R_h\left(\lambda I-S_h\right)-I\right\}\in B_{\infty }$, i.e.

\noindent 
\[\left(\lambda \dot{\left\{I\right\}}-\dot{\left\{S_h\right\}}\right)\dot{\left\{R_h\right\}}=\dot{\left\{(\lambda I-S_h)R_h\right\}}=\dot{\left\{I\right\}}\] 
and
\[\dot{\left\{R_h\right\}}\left(\lambda \dot{\left\{I\right\}}-\dot{\left\{S_h\right\}}\right)=\dot{\left\{R_h(\lambda I-S_h)\right\}}=\dot{\left\{I\right\}}.\] 
Therefore  $\lambda \in r\left(\dot{\left\{S_h\right\}}\right).$
\end{proof}
\noindent 
\\
\begin{remark} Let $\left\{S_h\right\}\ \in C_b\left(\left.(0,1\right],\ B\left(X\right)\right)$ and $\lambda \in r(\left\{S_h\right\})$\textit{.} Then there is $\left\{{\mathcal R}(\lambda ,S_h)\right\}\in C_b\left(\left.(0,1\right],\ B\left(X\right)\right)$ such that

\noindent 
\[{\ \mathop{\lim }_{h\to 0} \left\|\left(\lambda I-S_h\right){\mathcal R}\left(\lambda ,S_h\right)-I\right\|\ }={\mathop{\lim }_{h\to 0} \left\|{\mathcal R}\left(\lambda ,S_h\right)\left(\lambda I-S_h\right)-I\right\|\ }=0.\]

\noindent if and only if
\[\left(\lambda \dot{\left\{I\right\}}-\dot{\left\{S_h\right\}}\right)\dot{\left\{{\mathcal R}\left(\lambda ,S_h\right)\right\}}=\dot{\left\{I\right\}}=\dot{\left\{{\mathcal R}\left(\lambda ,S_h\right)\right\}}\left(\lambda \dot{\left\{I\right\}}-\dot{\left\{S_h\right\}}\right).\] 
 \end{remark}

\begin{proposition}
\label{d3.13} Let $\left\{S_h\right\},\ \left\{T_h\right\}\ \in C_b\left(\left.(0,1\right],\ B\left(X\right)\right)$ be two families. If ${\mathop{\lim}_{h\to 0} \left\|T_hS_h-S_hT_h\right\|\ }=0$, then $\ {\mathop{\lim}_{h\to 0} \left\|R(\lambda ,T_h)S_h-S_hR(\lambda ,T_h)\right\|\ }=0$, for any $\lambda \in r(\left\{T_h\right\})$.
\end{proposition}

\begin{proof}
 If $\lambda \in r(\left\{T_h\right\})$, then there is $\left\{{{\mathcal R}(\lambda ,T}_h)\right\}\in C_b\left(\left.(0,1\right],\ B\left(X\right)\right)$ such that

\noindent 
\[{\ \mathop{\lim }_{h\to 0} \left\|\left(\lambda I-T_h\right){{\mathcal R}(\lambda ,T}_h)-I\right\|\ }={\mathop{\lim }_{h\to 0} \left\|{{\mathcal R}(\lambda ,T}_h)\left(\lambda I-T_h\right)-I\right\|\ }=0.\] 
Therefore

\noindent 
\[{\mathop{\lim}_{h\to 0} \left\|T_hS_h-S_hT_h\right\|\ }=0\ \Leftrightarrow \dot{\left\{S_h\right\}}\dot{\left\{T_h\right\}}=\dot{\left\{T_h\right\}}\dot{\left\{S_h\right\}}\Leftrightarrow \]
\[\dot{\left\{S_h\right\}}\dot{\left\{{{\mathcal R}(\lambda ,T}_h)\right\}}=\dot{\left\{{{\mathcal R}(\lambda ,T}_h)\right\}}\dot{\left\{S_h\right\}}\ \Leftrightarrow {\mathop{lim}_{h\to 0} \left\|R(\lambda ,T_h)S_h-S_hR(\lambda ,T_h)\right\|\ }=0.\] 
\end{proof}

\begin{remark} i) Let $\left\{S_h\right\},\ \left\{T_h\right\}\ \in C_b\left(\left.(0,1\right],\ B\left(X\right)\right)$ such that $S_h$ is asymptotically equivalent with $T_h$, $\forall h\in \left.(0,1\right].$ Then 
\[Sp(T_h)=Sp(S_h),\forall h\in \left.(0,1\right].\]

\noindent ii) Let $\left\{S_h\right\},\ \left\{T_h\right\}\ \in C_b\left(\left.(0,1\right],\ B\left(X\right)\right)$ be asymptotically equivalent. Then

\noindent 
\[Sp({\{T}_h\})=Sp({\{S}_h\}).\] 
\end{remark}

\begin{theorem}
\label{d3.14} Let $\left\{S_h\right\},\ \left\{T_h\right\}\in \ C_b\left(\left.(0,1\right],\ B\left(X\right)\right)$ be two asymptotically quasinilpotent equivalent families. Then 
\[Sp({\{T}_h\})=Sp({\{S}_h\}).\] 
\end{theorem}

\begin{proof} Let  $\lambda \in r(\dot{{\{T}_h\}})$. Then there is  $\dot{\left\{{\mathcal R}(\lambda ,T_h)\right\}}\in B_{\infty }$ such that

\noindent 
\[\left(\lambda \dot{\left\{I\right\}}-\dot{{\{T}_h\}}\right){\rm \ }\dot{\left\{{\mathcal R}(\lambda ,T_h)\right\}}={\rm \ }\dot{\left\{{\mathcal R}(\lambda ,T_h)\right\}}\left(\lambda \dot{\left\{I\right\}}-\dot{{\{T}_h\}}\right)=\dot{\left\{I\right\}}.\]

\noindent Since $B_{\infty }$ is a Banach algebra, the map  $\lambda \mapsto \dot{\left\{{\mathcal R}(\lambda ,T_h)\right\}:}r(\dot{{\{T}_h\}})\to B_{\infty }$ is analytic.

\noindent 

\noindent Let $D_1=\left\{\left.\lambda \in {\mathbb C}\right|\left|\lambda -{\lambda }_0\right|\le r_1\right\}\subset r(\dot{{\{T}_h\}})$ and $D_0=\left\{\left.\lambda \in {\mathbb C}\right|\left|\lambda -{\lambda }_0\right|\le r_0\right\}$ with $r_1>r_0$.

\noindent 

\noindent Set
\[\dot{\left\{R_n\left(\lambda \right)\right\}}=\frac{1}{n!}\frac{d^n}{d{\lambda }^n}\dot{\left\{{\mathcal R}\left(\lambda ,T_h\right)\right\}}, \forall n\in {\mathbb N},\] 
and
\[\dot{\left\{R\left(\lambda \right)\right\}}=\sum_{n\in N}{\frac{{\left(-1\right)}^n}{n!}\dot{\left\{{\left(S_h-T_h\right)}^{\left[n\right]}\right\}}\dot{\left\{R_n(\lambda )\right\}}}.\]

\noindent If we set  $M_1={sup}_{\mu \in D_1}\left\|\dot{\left\{{\mathcal R}\left(\mu ,T_h\right)\right\}},\right\|$ , it follows that $\left\|\dot{\left\{R_n\left(\lambda \right)\right\}}\right\|\le \frac{r_1M_1}{{(r_1-r_0)}^{n+1}}$.

\noindent Deriving the relation $\left(\lambda \dot{\left\{I\right\}}-\dot{{\{T}_h\}}\right){\rm \ }\dot{\left\{{\mathcal R}(\lambda ,T_h)\right\}}=\dot{\left\{I\right\}}$ by $n$ times, we obtain

\noindent 
\[\left(\lambda \dot{\left\{I\right\}}-\dot{{\{T}_h\}}\right)\frac{d^n}{d{\lambda }^n}\dot{\left\{{\mathcal R}\left(\lambda ,T_h\right)\right\}}=-n\frac{d^{n-1}}{d{\lambda }^{n-1}}\dot{\left\{{\mathcal R}\left(\lambda ,T_h\right)\right\}}.\]

\noindent Moreover, since

\noindent 
\[\dot{\left\{{\left(T_h-S_h\right)}^{\left[n+1\right]}\right\}}=\dot{\left\{T_h{\left(T_h-S_h\right)}^{\left[n\right]}-{\left(T_h-S_h\right)}^{\left[n\right]}S_h\right\}}=\]
\[=\dot{\left\{T_h\right\}}\dot{\left\{{\left(T_h-S_h\right)}^{\left[n\right]}\right\}}-\dot{\left\{{\left(T_h-S_h\right)}^{\left[n\right]}\right\}}\dot{\left\{S_h\right\}}, \forall n\in {\mathbb N}, \] 
we have
\[\left(\lambda \dot{\left\{I\right\}}-\dot{\left\{S_h\right\}}\right)\dot{\left\{R\left(\lambda \right)\right\}}=\dot{\left\{\lambda I-S_h\right\}}\sum_{n\in N}{\frac{{\left(-1\right)}^n}{n!}\dot{\left\{{\left(S_h-T_h\right)}^{\left[n\right]}\right\}}\dot{\left\{R_n\left(\lambda \right)\right\}}}=\]
\[=\sum_{n\in N}{\frac{{\left(-1\right)}^n}{n!}\dot{\left\{\lambda I-S_h\right\}}\dot{\left\{{\left(S_h-T_h\right)}^{\left[n\right]}\right\}}\dot{\left\{R_n\left(\lambda \right)\right\}}}=\]
\[=\sum_{n\in N}{\frac{{\left(-1\right)}^n}{n!}\dot{\left\{(\lambda I-S_h{)\left(S_h-T_h\right)}^{\left[n\right]}\right\}}\dot{\left\{R_n\left(\lambda \right)\right\}}}=\]
\[=\sum_{n\in N}{\frac{1}{n!}\dot{\left\{\left({\left(\left(\lambda I-S_h\right)-(\lambda I-T_h\right))}^{\left[n+1\right]}+{\left(\left(\lambda I-S_h\right)-(\lambda I-T_h\right))}^{\left[n\right]}\left(\lambda I-T_h\right)\right)\right\}}\dot{\left\{R_n\left(\lambda \right)\right\}}}=\]
\[=\sum_{n=0}{\frac{{\left(-1\right)}^{n+1}}{n!}\dot{\left\{{\left(S_h-T_h\right)}^{\left[n+1\right]}\right\}}\dot{\left\{R_n(\lambda )\right\}}}+\left(\lambda \dot{\left\{I\right\}}-\dot{\left\{T_h\right\}}\right)\dot{\left\{{\mathcal R}\left(\lambda ,T_h\right)\right\}}-\]
\[-\sum_{n=1}{\frac{{\left(-1\right)}^n}{n-1!}\dot{\left\{{\left(S_h-T_h\right)}^{\left[n\right]}\right\}}\left(\lambda \dot{\left\{I\right\}}-\dot{{\{T}_h\}}\right)\frac{d^n}{d{\lambda }^n}\dot{\left\{{\mathcal R}\left(\lambda ,T_h\right)\right\}}}=\]
\[=\sum_{n=0}{\frac{{\left(-1\right)}^{n+1}}{n!}\dot{\left\{{\left(S_h-T_h\right)}^{\left[n+1\right]}\right\}}\dot{\left\{R_n(\lambda )\right\}}}+\left(\lambda \dot{\left\{I\right\}}-\dot{\left\{T_h\right\}}\right)\dot{\left\{{\mathcal R}\left(\lambda ,T_h\right)\right\}}-\]
\[-\ \sum_{n=1}{\frac{{\left(-1\right)}^n}{n-1!}\dot{\left\{{\left(S_h-T_h\right)}^{\left[n\right]}\right\}}\frac{d^{n-1}}{d{\lambda }^{n-1}}\dot{\left\{{\mathcal R}\left(\lambda ,T_h\right)\right\}}}=\]
\[=\sum_{n=0}{\frac{{\left(-1\right)}^{n+1}}{n!}\dot{\left\{{\left(S_h-T_h\right)}^{\left[n+1\right]}\right\}}\dot{\left\{R_n(\lambda )\right\}}}+\left(\lambda \dot{\left\{I\right\}}-\dot{\left\{T_h\right\}}\right)\dot{\left\{{\mathcal R}\left(\lambda ,T_h\right)\right\}}-\]
\[-\sum_{n=1}{\frac{{\left(-1\right)}^n}{n-1!}\dot{\left\{{\left(S_h-T_h\right)}^{\left[n\right]}\right\}}\dot{\left\{R_{n-1}(\lambda )\right\}}}.\]

\noindent Therefore $\lambda \in r({\{S}_h\})$.

\noindent Analogously we can prove the other inclusion. By  Theorem \ref{d3.12}, it results that

\noindent 
\[Sp({\{T}_h\})=Sp(\dot{{\{T}_h\}})=Sp(\dot{{\{S}_h\}})=Sp({\{S}_h\}).\] 
\end{proof}

\begin{theorem}
\label{d3.15} Let $\left\{T_h\right\}\ \in C_b\left(\left.(0,1\right],\ B\left(X\right)\right)$ and $\Omega $ be an open set which contains $\bigcup_{h\in \left.(0,1\right]}{Sp(T_h)}$. Then for any analytic function $f:\Omega \to {\mathbb C}$ we have

\noindent 
\[Sp\left(\left\{f\left(T_h\right)\right\}\right)=f(Sp\left(\left\{T_h\right\}\right)).\] 
\end{theorem}

\begin{proof}
 If $\left\{T_h\right\}\in C_b\left(\left.(0,1\right],\ B\left(X\right)\right)$, then there is a $M<\infty $ such that $\left\|T_h\right\|\le M$, $\forall h\in \left.(0,1\right]$. Therefore $Sp(T_h)\subset D(0,M)$ $\forall h\in \left.(0,1\right]$, so that $\bigcup_{h\in \left.(0,1\right]}{Sp(T_h)}$ is a bounded set.

\noindent ''$\supseteq $'' Let $f:\Omega \to {\mathbb{C}}$ be an analytic
function and $\lambda \in Sp(\left\{T_h\right\})$. For $\xi \in \Omega $, we
define the function 
\begin{equation*}
g\left(\xi \right)=\left\{ 
\begin{array}{c}
\frac{f\left(\xi \right)-f(\lambda )}{\xi -\lambda },\ \xi \ne \lambda \\ 
f^{\prime }\left(\lambda \right),\ \xi =\lambda%
\end{array}
\right..
\end{equation*}

\noindent Hence $g:\Omega \to {\mathbb{C}}$\textit{\ }is analytic and

\noindent 
\begin{equation*}
f\left(T_h\right)-\ f\left(\lambda \right)I=\
g\left(T_h\right)\left(T_h-\lambda I\right)=\left(T_h-\lambda
I\right)g\left(T_h\right),
\end{equation*}

\noindent for any $h\in \left.(0,1\right]$.

\noindent We suppose that $f\left(\lambda \right)\in
r\left(\left\{f\left(T_h\right)\right\}\right)$. Then there is $\left\{{%
\mathcal{R}}(f\left(\lambda \right),\ f\left(T_h\right))\right\}\ \subset
B(X)$ such that

\noindent 
\begin{equation*}
{\mathop{\lim }_{h\to 0} \left\|\left(f\left(\lambda \right)I-\
f\left(T_h\right)\right){\mathcal{R}}\left(f\left(\lambda \right),\
f\left(T_h\right)\right)-I\right\|\ }=
\end{equation*}
\begin{equation*}
={\mathop{\lim }_{h\to 0} \left\|{\mathcal{R}}\left(f\left(\lambda \right),\
f\left(T_h\right)\right)\left(f\left(\lambda \right)I-\
f\left(T_h\right)\right)-I\right\|\ }=0.
\end{equation*}

\noindent Having in view the last relation, we have

\noindent 
\begin{equation*}
{\mathop{\lim }_{h\to 0} \left\|\left(T_h-\lambda I\right)g\left(T_h\right){%
\mathcal{R}}\left(f\left(\lambda \right),\
f\left(T_h\right)\right)-I\right\|\ }=
\end{equation*}
\begin{equation*}
={\mathop{\lim }_{h\to 0} \left\|{\mathcal{R}}\left(f\left(\lambda \right),\
f\left(T_h\right)\right)g\left(T_h\right)\left(T_h-\lambda
I\right)-I\right\|\ }=0. (*)
\end{equation*}
Since 
\begin{equation*}
\ g\left(T_h\right)T_h=T_hg\left(T_h\right),
\end{equation*}

\noindent for any $h\in \left.(0,1\right]$, according to the properties of
holomorphic functional calculi it follows

\noindent 
\begin{equation*}
g\left(T_h\right)f(T_h)=f(T_h)g\left(T_h\right),
\end{equation*}

\noindent for every $h\in \left.(0,1\right]$. Applying Proposition \ref{d3.11}, we
obtain

\noindent 
\begin{equation*}
{\mathop{\lim }_{h\to 0} \left\|g\left(T_h\right){\mathcal{R}}%
\left(f\left(\lambda \right),\ f\left(T_h\right)\right)-{\mathcal{R}}%
\left(f\left(\lambda \right),\
f\left(T_h\right)\right)g\left(T_h\right)\right\|\ }=0.
\end{equation*}
Hence

\noindent 
\begin{equation*}
{\mathop{\lim }_{h\to 0} \left\|g\left(T_h\right){\mathcal{R}}%
\left(f\left(\lambda \right),\ f\left(T_h\right)\right)\left(T_h-\lambda
I\right)-I\right\|\ }=
\end{equation*}
\begin{equation*}
={\mathop{\lim }_{h\to 0} \|g\left(T_h\right){\mathcal{R}}%
\left(f\left(\lambda \right),\ f\left(T_h\right)\right)\left(T_h-\lambda
I\right)-{\mathcal{R}}\left(f\left(\lambda \right),\
f\left(T_h\right)\right)g\left(T_h\right)\left(T_h-\lambda I\right)}+
\end{equation*}
\begin{equation*}
+{{\mathcal{R}}\left(f\left(\lambda \right),\
f\left(T_h\right)\right)g\left(T_h\right)\left(T_h-\lambda I\right)-I \|}\le
\end{equation*}
\begin{equation*}
\le {\mathop{\lim }_{h\to 0} \left\|g\left(T_h\right){\mathcal{R}}%
\left(f\left(\lambda \right),\ f\left(T_h\right)\right)-{\mathcal{R}}%
\left(f\left(\lambda \right),\
f\left(T_h\right)\right)g\left(T_h\right)\right\|\left\|T_h-\lambda
I\right\|\ }+
\end{equation*}
\begin{equation*}
+{\mathop{\lim }_{h\to 0} \left\|{\mathcal{R}}\left(f\left(\lambda \right),\
f\left(T_h\right)\right)g\left(T_h\right)\left(T_h-\lambda
I\right)-I\right\|\ }=0. (**)
\end{equation*}

\noindent From (*) and (**), it results

\noindent 
\begin{equation*}
{\mathop{\lim }_{h\to 0} \left\|\left(T_h-\lambda I\right)g\left(T_h\right){%
\mathcal{R}}\left(f\left(\lambda \right),\
f\left(T_h\right)\right)-I\right\|\ }=
\end{equation*}
\begin{equation*}
={\mathop{\lim }_{h\to 0} \left\|g\left(T_h\right){\mathcal{R}}%
\left(f\left(\lambda \right),\ f\left(T_h\right)\right)\left(T_h-\lambda
I\right)-I\right\|\ }=0,
\end{equation*}

\noindent so $\lambda \in r\left(\left\{T_h\right\}\right)$, contradiction
with $\lambda \in Sp\left(\left\{T_h\right\}\right)$. Therfore $%
f\left(\lambda \right)\in \ Sp\left(\left\{f(T_h)\right\}\right)$.

\noindent ''$\subseteq $'' Let $\lambda \in \
Sp\left(\left\{f(T_h)\right\}\right)$. If $\lambda \notin
f(Sp\left(\left\{T_h\right\}\right))$, then $\lambda \ne f(\xi )$ for any $%
\xi \in Sp\left(\left\{T_h\right\}\right)$.

\noindent Let $\Omega ^{\prime }$ an open neighborhood $\bigcup_{h\in \left.(0,1\right]}{Sp(T_h)}$ and

\noindent 
\begin{equation*}
h\left(\xi \right)=\frac{1}{f\left(\xi \right)-\lambda } ,
\end{equation*}

\noindent for every $\xi \in \Omega ^{\prime }$. Then $h$ is an analytic
function and applying the holomorphic functional calculi, we obtain 
\begin{equation*}
h\left(T_h\right)\left(f\left(T_h\right)-\lambda
I\right)=\left(f\left(T_h\right)-\lambda I\right)h\left(T_h\right)=I,
\end{equation*}

\noindent for any $h\in \left.(0,1\right]$. Therefore $\lambda \in r(f(T_h))$%
, for any $h\in \left.(0,1\right]$. Since $\bigcap_{h\in (\left.0,1\right]}{%
r\left(f(T_h)\right)}\subseteq r\left(\left\{f(T_h)\right\}\right)$ (Remark
3.2 i)), it follows $\lambda \in r\left(\left\{f(T_h)\right\}\right)$,
contradiction with $\lambda \in Sp\left(\left\{{f(T}_h)\right\}\right)$.
Hence $\lambda \in f(Sp\left(\left\{T_h\right\}\right))$.

\noindent
\end{proof}

\begin{definition}
\label{d3.16} A family $\left\{U_h\right\}\subset L(X)$ is
calling asymptotic quasinilpotent operator if

\noindent 
\begin{equation*}
{\mathop{\lim}_{n\to \infty } {{\mathop{\lim \sup}_{h\to 0} \left\|{U_h}^n\right\|\ 
}}^{\frac{1}{n}}\ }=0.
\end{equation*}
\end{definition}

\begin{theorem}
\label{d3.17} A family $\left\{U_h\right\}\in C_b\left(\left.(0,1\right],\ B\left(X\right)\right)$ is an asymptotic quasinilpotent operator if and only if $Sp\left(\left\{U_h\right\}\right)=\left\{0\right\}$.
\end{theorem}
\noindent 

\begin{proof} Let $\left\{U_h\right\}\in C_b\left(\left.(0,1\right],\ B\left(X\right)\right)$ be an asymptotic quasinilpotent operator. Then $\left\{U_h\right\}$ is asymptotically spectral equivalent with ${\left\{0\right\}}_{h\in \left.(0,1\right]}\in C_b\left(\left.(0,1\right],\ B\left(X\right)\right)$. By Theorem \ref{d3.14} it follows that
\[Sp\left(\left\{U_h\right\}\right)=Sp\left(\left\{0\right\}\right)=\left\{0\right\}.\] 

\noindent Consequently, suppose that  $Sp\left(\left\{U_h\right\}\right)=\left\{0\right\}$. By Theorem \ref{d3.12}, we have

\noindent 
\[Sp\left(\dot{\left\{U_h\right\}}\right)=Sp\left(\left\{U_h\right\}\right)=\left\{0\right\}.\]

\noindent Then the spectral radius of $\dot{\left\{U_h\right\}}$, which we will call from now $r_{sp}\left(\dot{\left\{U_h\right\}}\right)$, is zero. Since

\noindent 
\[r_{sp}\left(\dot{\left\{U_h\right\}}\right)={\mathop{\lim }_{n\to \infty } {\left\|{\left(\dot{\left\{U_h\right\}}\right)}^n\right\|}^{\frac{1}{n}}\ },\] 
it follows that 
\[{\mathop{\lim }_{n\to \infty } {\left\|{\left(\dot{\left\{U_h\right\}}\right)}^n\right\|}^{\frac{1}{n}}\ }=0.\] 
But, on the other hand, we have

\noindent 
\[{\mathop{\lim }_{n\to \infty } {\left\|{\left(\dot{\left\{U_h\right\}}\right)}^n\right\|}^{\frac{1}{n}}\ }={\mathop{\lim }_{n\to \infty } {inf}_{\left\{U_h\right\}\in \dot{\left\{U_h\right\}}}{\left\|{\left\{U_h\right\}}^n\right\|}^{\frac{1}{n}}\ }={\mathop{\lim }_{n\to \infty } {inf}_{\left\{U_h\right\}\in \dot{\left\{U_h\right\}}}{\left\|\left\{{U_h}^n\right\}\right\|}^{\frac{1}{n}}\ }=\]
\[={\mathop{\lim }_{n\to \infty } {inf}_{\left\{U_h\right\}\in \dot{\left\{U_h\right\}}}{{sup}_{h\in \left.(0,1\right]}\left\|{U_h}^n\right\|}^{\frac{1}{n}}\ }\ge {\mathop{\lim }_{n\to \infty } {inf}_{\left\{U_h\right\}\in \dot{\left\{U_h\right\}}}{{\mathop{{\lim \sup}}_{h\to 0} \left\|{U_h}^n\right\|\ }}^{\frac{1}{n}}\ }=\]
\[={\mathop{\lim }_{n\to \infty } {{\mathop{{\lim \sup}}_{h\to 0} \left\|{U_h}^n\right\|\ }}^{\frac{1}{n}}\ }.\]

\noindent By the above relations, we obtain

\noindent 
\[{\mathop{\lim }_{n\to \infty } {{\mathop{{\lim \sup}}_{h\to 0} \left\|{U_h}^n\right\|\ }}^{\frac{1}{n}}\ }=0,\]

\noindent so that $\left\{U_h\right\}$ is an asymptotic quasinilpotent operator. 
\end{proof}

\end{document}